\documentclass[10pt,leqno,twoside]{amsart}

\usepackage{enumerate}

\usepackage{amssymb, verbatim}

\usepackage{amssymb,amsthm,amsmath,eucal,mathrsfs}

\usepackage{tikz, subfigure, xcolor} 

\usepackage{pgfplots}

\usepackage{cite}

\usepackage{amsmath}

\usepackage{amsthm}

\usepackage{amssymb}

\usepackage{amsfonts}

\usepackage{dsfont}

\usepackage{hyperref}

\newtheorem{theorem}{Theorem}[section]

\newtheorem{lemma}[theorem]{Lemma}

\theoremstyle{definition}

\newtheorem{definition}[theorem]{Definition}

\newtheorem{remark}[theorem]{Remark}

\numberwithin{equation}{section}

\newcommand\restr[2]{{
  \left.\kern-\nulldelimiterspace 
  #1 
  \vphantom{\big|} 
  \right|_{#2} 
  }}

\usepackage{eucal}

\title[Imaging using nanoparticles as contrast agents]{Mathematical imaging using electric or magnetic nanoparticles as contrast agents}
\author[Challa, Choudhury, Sini]{Durga Prasad \textbf{C}halla $^* $, Anupam Pal \textbf{C}houdhury $^{\dag} $, Mourad \textbf{S}ini $^{\ddag} $} 

\keywords{Imaging, nanoparticles, asymptotics, integral equations, electromagnetism.\\
$^* $ Faculty of Mathematics, Indian Institute of Technology Tirupati, Tirupati, India. Email: chsmdp@iittp.ac.in. This author was partially supported by the Austrian Science Fund (FWF): P28971-N32.\\
$^{\dag}$ Technische Universit{\"a}t Darmstadt, Institute of Mathematics, Schlossgartenstr. 7, 
64289 Darmstadt, Germany. Email: anupampcmath@gmail.com.  This author is supported by the DFG International Research Training Group IRTG 1529 on Mathematical Fluid Dynamics at TU Darmstadt. \\ 
$^{\ddag}$ RICAM, Austrian Academy of Sciences,
Altenbergerstrasse 69, A-4040, Linz, Austria.
Email:mourad.sini@oeaw.ac.at. This author is partially supported by the Austrian Science Fund (FWF): P28971-N32}

\begin{document}

\maketitle

\begin{abstract}  
We analyse mathematically the imaging modality using electromagnetic nanoparticles as contrast agent. 
This method uses the electromagnetic fields, collected before and after injecting electromagnetic nanoparticles, to reconstruct the electrical permittivity.
The particularity here is that these nanoparticles have high contrast electric or magnetic properties compared to
the background media. First, we introduce the concept of electric (or magnetic) nanoparticles to describe the particles, of relative diameter $\delta$ 
(relative to the size of the imaging domain), having relative electric permittivity (or relative magnetic permeability) of order $\delta^{-\alpha}$ with a certain $\alpha>0$, as $0<\delta<<1$.
Examples of such material, used in the imaging community, are discussed. 
Second, we derive the asymptotic expansion of the electromagnetic fields due to such singular contrasts. We consider here the scalar electromagnetic model. 
Using these expansions, we extract the values of the total fields inside the domain of imaging from the scattered fields measured before and after 
injecting the nanoparticles. From these total fields, we derive the values of the electric permittivity at the expense of numerical differentiations.
\end{abstract}

\section{Introduction}

 Let $D$ be a small body in $\mathbb{R}^{3}$ of the form $D:= \delta B+z$, where
 $B$ is open, bounded with a Lipschitz boundary, simply connected set in $\mathbb{R}^{3}$ containing the origin, and $z$ specify the locations of $D$. 
 The parameter $\delta > 0 $ is the $B$-relative radius of $D$. It characterises the smallness assumption on the body $D$ as compared to $B$.\\
Let us consider the magnetic permeability $\mu_\delta$ and electric permittivity $\epsilon_\delta$  of the form 
\begin{equation}
\mu_{\delta}(x)=\begin{cases}
                 \mu_{0}, \ x \in \mathbb{R}^{3}\diagdown \overline D,\\
                 \mu_{1}, \ x \in D, \end{cases}
\label{model1}
\end{equation}
%
%
\begin{equation}
\epsilon_{\delta}(x)=\begin{cases}
                \epsilon_{0}, \ x \in \mathbb{R}^{3}\diagdown \overline D,\\
                \epsilon_{1}, \ x \in D,
               \end{cases}
\label{model2}
\end{equation}
where $\mu_{0},\mu_{1}, \epsilon_{1}$ are positive constants while $\epsilon_0:=\epsilon_0(x)$ is variable in a bounded domain, i.e. there exists $\Omega$ bounded 
such that $\epsilon_0$ is a constant in $\mathbb{R}^{3}\diagdown \overline \Omega$. Here, $\Omega$ is of the same order as the reference body $B$. Thus $\mu_{0},\epsilon_{0}$ denote the permeability and permittivity of the 
background medium and $\mu_{1},\epsilon_{1}$ denote the permeability and permittivity of the scatterers respectively.\\
We are then interested in the following scattering problem:
\begin{equation}
\begin{cases}   \nabla. (\frac{1}{\mu_{0}} \nabla V) +\omega^{2} \epsilon_{0} V =0 \ \text{in} \ \mathbb{R}^{3}\diagdown \overline D,\\
   \nabla. (\frac{1}{\mu_{1}} \nabla V) +\omega^{2} \epsilon_{1} V =0 \ \text{in}\ D,\\
   \left.V \right\vert_{-}-\left.V \right\vert_{+}=0, \ \text{on}\ \partial D,\\
   \frac{1}{\mu_{1}}\left.\frac{\partial V}{\partial \nu} \right\vert_{-}-\frac{1}{\mu_{0}} \left.\frac{\partial V}{\partial \nu} \right\vert_{+} =0\ \text{on} \ \partial D,
   \end{cases}
\label{model3}
\end{equation}
where $\omega > 0$ is a given frequency. The total field $V$ has the form $V:=V^{I}+V^{s}$ where $V^{I}$ denotes the incident field and $V^{s}$ denotes the scattered waves. The above set of equations have to be supplemented with the \textit{Sommerfeld radiation
condition} on $V^{s}$ which we shall henceforth refer to as $(S.R.C)$. Here, $V$ describes components of the electric field \footnote{ In case we have invariance of the model in one direction, 
this Helmholtz model describes the propagation of the component of the electric field which is orthogonal to that axis. To simplify the exposition, we stated it in $3D$ instead of $2D$ to avoid 
the log-type singularities of the corresponding Green's functions.}. Keeping in mind the positivity of the permeabilities, the above problem can be equivalently formulated as 
\begin{equation}
 \begin{cases}
  \Delta V + \kappa_{0}^{2} V =0 \ \text{in} \ \mathbb{R}^{3}\diagdown \overline D, \\
   \Delta V + \kappa_{1}^{2} V =0 \ \text{in}\ D,\\
   \left.V \right\vert_{-}-\left.V \right\vert_{+}=0, \ \text{on} \ \partial D,\\
   \frac{1}{\mu_{1}}\left.\frac{\partial V}{\partial \nu^{}} \right\vert_{-}-\frac{1}{\mu_{0}} \left.\frac{\partial V}{\partial \nu^{}} \right\vert_{+} =0\ \text{on} \ \partial D,\\
   \frac{\partial V^{s}}{\partial \vert x \vert}-i \kappa_{0} V^{s} =o (\frac{1}{\vert x \vert}) , \ \vert x \vert \rightarrow \infty \ (S.R.C),
 \end{cases}
\label{model4-}
\end{equation}
where $\kappa_{0}^{2}=\omega^{2} \epsilon_{0} \mu_{0}$, $\kappa_{1}^{2}=\omega^{2} \epsilon_{1} \mu_{1}$ and $V^{s}$ denotes the scattered field. We restrict to plane incident waves $V^I $, i.e. $V^{I}(x):=e^{i\kappa_0 x \cdot d} $, where $x \in \mathbb{R}^3, \ d\in \mathbb{S}^{2} $. Here $\mathbb{S}^{2} $ denotes the unit sphere in $\mathbb{R}^3 $.\\

The scattering problem 
(\ref{model4-}) is well posed in appropriate spaces and the scattered field $V^s(x, d)$ has the following asymptotic expansion:
\begin{equation}\label{far-field}
 V^s(x, d)=\frac{e^{i \kappa |x|}}{|x|}V^{\infty}(\hat{x}, d) + \mathcal{O}(|x|^{-2}), \quad |x|
\rightarrow \infty,
\end{equation}
with $\hat{x}:=\frac{x}{\vert x\vert}$, where the function
$V^{\infty}(\hat{x}, d)$ for $(\hat{x}, d)\in \mathbb{S}^{2} \times \mathbb{S}^{2}$  is called the far-field pattern.\\
{\bf{Problem}}. Our motivation in this work is the use of electromagnetic nanoparticles as contrast agent to image the background, i.e. to reconstruct the coefficient $\epsilon_0$ in $\Omega$. 
The measured data in this case is described as follows: 
\begin{enumerate}
 \item $V^0:=\{V^{\infty}(-d, d),\; \mbox{for a single direction } d \}$ when no nanoparticle is injected.
 \item  $V^1:=\{V^{\infty}(-d, d),\; \mbox{for a single direction } d\}$ when one single particle is injected to a given point $z\in \Omega$. 
 This measurement is needed for all the points (or a sample of points) $z$ in $\Omega$.
 \end{enumerate}
If, in (\ref{model3}) or (\ref{model4-}), we exchange the roles of $(\mu_1, \mu_0)$ with $(\epsilon_1, \epsilon_0)$, where $\mu_0$ is now locally variable while $\epsilon_0$ is a constant, 
then $V$ describes components of the magnetic field. For such a model, the goal is to reconstruct the magnetic permeabilty $\mu_0$ in $\Omega$ from the corresponding measured data generated by the nanoparticles
\footnote{ As the humain tissue is nonmagnetic, i.e. the corresponding permeability 
$\mu_0$ is merely constant, this model might not be useful for such imaging purpose. Nevertheless, we include this model here for the sake of completing the analysis.}.

\bigskip

Imaging using electromagnetic nanoparticles as contrast agents has drawn a considerable attention in the very recent years, see for instance \cite{B-B:2011,Chen-Craddock-Kosmas, Shea-Kosmas-VanVeen-Hagness}. 
To motivate it, let us first recall that conventional 
imaging techniques, such as microwave imaging techniques, are known to be potentially capable of extracting features in breast cancer, for instance, in case of relatively high contrast of
the permittivity and conductivity, between healthy tissues and malignant ones, \cite{F-M-S:2003}. However, it is observed that in the case of a benign tissue, the variation of the permittivity is quite low so that such conventional
imaging modalities are limited to be used for early detection of such diseases. 
In such cases, creating such missing contrast is highly desirable. One way to do it is to use electromagnetic nanoparticles as contrast agents, \cite{B-B:2011,Chen-Craddock-Kosmas}. 
In the literature, different types of nanoparticles have been proposed with this application in mind. Let us cite few of them: 
\begin{enumerate}
\item To create contrast in the permittivity, carbon nanotubes, ferroelectric nanoparticles and Calcium copper titanate are used, see \cite{B-B:2011}. 
Such particles have diameter which is estimated around $10$ nm, or $10^{-8}$ m, and have relative electrical permittivity of the order $10$ for the carbon nanotubes, $10^3$ for ferroelectric nanoparticles 
and around $10^6$ for Calcium copper titanate, see \cite{W1}. If the benign tumor 
is located at the cell level (which means that our $\Omega$ is that cell), with diameter of order $10^{-5}$ m, then the 
$\Omega$-relative radius of the particles are $\delta=10^{-3}$. Hence, the relative permittivity of the types of nanoparticles are estimated of the order $\delta^{-\alpha}$ where $\alpha$ is $\frac{1}{3}$ 
for the carbon nanotubes, $1$ for ferroelectric nanoparticles and around $2$ for Calcium copper titanate.

\item  The human tissue is known to be nonmagnetic. To create magnetic contrasts, it was also proposed in \cite{B-B:2011} to use iron oxide magnetic nanoparticles for imaging early tumors. Such material has relative 
 permeability of the order between $10^{4}$ and $10^{6}$, see \cite{Hyper-physcs}, and hence of the order $\delta^{-\alpha}$ with $\alpha$ between $\frac{4}{3}$ and $2$.
 
\item  If the tumor is a relatively malignant one and occupy bodies having diameter of order of few millimeters or even centimeters, 
then we end up with values of $\alpha$ of order $\frac{1}{5}, \frac{3}{5}$ and $1$ respectively for the aforementioned electric nanoparticles and between $\frac{2}{3}$ and $1$ for the iron oxide magnetic nanoparticles.
\end{enumerate}

\bigskip

This shows that for the detection of the tumors using such nanoparticles, we can model
the ratio of the relative electric permittivity and relative magnetic permeability as follows: 

\begin{definition}
We shall call $(D, \epsilon_1, \mu_1)$ an electromagnetic nanoparticle of shape $\partial D$, diameter of order $O(\delta), \delta<<1,$  and  permittivity and permeability $\epsilon_1, \mu_1$ respectively. 
\begin{enumerate}
 \item The particle $(D, \epsilon_1, \mu_1)$ is called an electric nanoparticle, relative to the background ($\epsilon_0, \mu_0$), if in addition we have 
 $\frac{\epsilon_1}{\epsilon_0}\sim \delta^{-\alpha}, \alpha>0, (\mbox{ i.e. } \frac{\epsilon_1}{\epsilon_0} >>1$) keeping the permeability moderate, i.e. $\frac{\mu_1}{\mu_0} \sim 1 $. In this case, the relative speed of propagation satisfies
 $\frac{\kappa^2_{1}}{\kappa^2_0}:= \frac{\epsilon_{1} \mu_{1}}{\epsilon_{0} \mu_{0}} \sim \delta^{-\alpha}$. 
 \item The particle $(D, \epsilon_1, \mu_1)$ is called a magnetic nanoparticle, relative to the background ($\epsilon_0, \mu_0$) if in addition we have $\frac{\mu_1}{\mu_0}\sim 
 \delta^{-\alpha}, \alpha>0, (\mbox{ i.e. } \frac{\mu_1}{\mu_0}>>1)$ keeping the permittivity moderate, i.e. $\frac{\epsilon_l}{\epsilon_0}\sim 1 $. Again in this case, the relative speed of propagation satisfies
 $\frac{\kappa^2_{1}}{\kappa^2_0} \sim \delta^{-\alpha}$.
\end{enumerate}
Here we use the notation $a \sim b $ to imply that there exists a positive constant $C $ such that $ C^{-1} \vert b \vert \leq \vert a \vert \leq C \vert b \vert $.
\end{definition}
Let us stress here that these electromagnetic nanoparticles have large relative speeds of propagation (or indices of refraction).\\
To highlight the relevance of each type of data that we use in stating the imaging problem, we briefly analyse their importance:
\begin{enumerate}
 \item {\it{Use of scattered data before injecting any nanoparticle.}} Assume that we have at hand the farfields $V^{\infty}(\hat{x}, d)$ in the case when $\hat{x}$ and $d$ are taken in the
whole $ \mathbb{S}^{2}$. The problem with such data is well
studied and there are  several algorithms to solve it , see \cite{Nachman:1988, Novikov:1988, Ramm:1988}. It is also known that this problem is 
very unstable. Precisely, the modulus of continuity is in general of logarithmic type, see \cite{MRAMAG-book1}.  

\item {\it{Use of scattered data before and after injecting one single nanoparticle at time.}}  In \cite{A-B-C-T-F:2008, A-C-D-R-T}, and the references therein, an approach of solving the problem using data 
collected before and after localized elastic perturbation has been proposed. 
However, their analysis goes also for the case of injecting a single nanoparticle at a time (precisely moderate nanoparticles, i.e.
both the electrical permittivity and the magnetic permeability are moderate in terms of the size of the particles (taking $\alpha=0$ above) ). 
It is divided into two steps. In the first one, from these far-fields, we can extract the total energies due to the medium $\epsilon_0$ in the 
interior of the support of $(\epsilon_0-1)$.  In the second step, we reconstruct $\epsilon_0(x)$ from 
these interior values of the energies. In contrast to 
the instability of the classical inverse scattering problem, the reconstruction from 
internal measurements is stable, see \cite{Alessandrini:2014, A-G-W:2012, H-M-N:2014, Triki:2010}. 
However, one needs to handle the invertibility properties of the energy mappings. One can think of using measurements related to multiple 
frequencies $\omega_1,\omega_2, etc.$, see \cite{Alberti:2013}, to invert such mappings.\\ 
Let us emphasize here that if we use electric (or magnetic) nanoparticles, instead of moderate nanoparticles, then we can extract, from the corresponding far-fields, the pointwise values
of the total fields and not only the energies, see Theorem \ref{Main-theorem-anal}. This is one reason why we introduced the concept of electric (or magnetic) nanoparticles. The second step is then
to extract the coefficients $\epsilon_0$ from these data. This can be done at the expense of a numerical differentiation. It is known that numerical differentiation is an unstable step but it is only moderately unstable.
Finally, observe that, we use the backscattered data, before and after injecting the particles, in only one and arbitrary direction $d$.
\end{enumerate}

We found that using electric nanoparticles (see the definition above), the electric far-field data are enough for the reconstruction. However, if we use magnetic nanoparticles, then, in addition to the electric far-fields, 
the electric near-fields are also required.
Similarly, using magnetic nanoparticles, the magnetic far-field data are enough for the reconstruction. However, if we use electric nanoparticles, then, in addition to the magnetic far-fields, 
the magnetic near-fields are also needed. 
Hence, as a moral, if we use electric (resp. magnetic) nanoparticles, it is better to use electric (resp. magnetic) far-fields. 
In both cases, the error of the reconstruction is fixed by the parameter $\alpha$ which is related to the kind of (relative permittivity or permeability of the)  nanoparticles used. 
In addition, the distance from the place where to collect the near-field data to the tumor is also fixed by the kind of the nanoparticles used. More details are provided in Remark \ref{Remark-on-measured-data}.
\bigskip

The remaining part of the paper is organized as follows. In Section $2$, we state the asymptotic expansion of the far-fields and near-fields generated by high contrast electric permittivity or magnetic permeability.
Then, we discuss how these expansions can be used to derive the needed formulas to solve the imaging problem. In section $3$, we give the details of the proof of the main theorem stated in section $2$.

\section{The asymptotic expansion of the fields and application to imaging}
\subsection{The asymptotic expansion with singular coefficients}
Let us recall our model:
\begin{equation}
 \begin{cases}
  \Delta V + \kappa_{0}^{2} V =0 \ \text{in} \ \mathbb{R}^{3}\diagdown \overline D, \\
   \Delta V + \kappa_{1}^{2} V =0 \ \text{in}\ D,\\
   \left.V \right\vert_{-}-\left.V \right\vert_{+}=0, \ \text{on} \ \partial D,\\
   \frac{1}{\mu_{1}}\left.\frac{\partial V}{\partial \nu^{}} \right\vert_{-}-\frac{1}{\mu_{0}} \left.\frac{\partial V}{\partial \nu^{}} \right\vert_{+} =0\ \text{on} \ \partial D,\\
   \frac{\partial V^{s}}{\partial \vert x \vert}-i \kappa_{0} V^{s} =o (\frac{1}{\vert x \vert}) , \ \vert x \vert \rightarrow \infty \ (S.R.C),
 \end{cases}
\label{model4}
\end{equation}
where $\kappa_{0}^{2}=\omega^{2} \epsilon_{0} \mu_{0}$, $\kappa_{1}^{2}=\omega^{2} \epsilon_{1} \mu_{1}$ and $V^{s}$ denotes the scattered field.\\ 

In the case when $\epsilon_0$ is a constant every where in $\mathbb{R}^{3}$, we use the notation $u$ instead of $V$ in \eqref{model4}. The solution $u$ can then be represented as
\begin{equation}
u(x)=\begin{cases}
      u^{I}+ S^{\kappa_{0}}_{D} \phi(x), \ x \in \mathbb{R}^{3}\diagdown \overline D,   \\
       S^{\kappa_{1}}_{D} \psi(x) , \ x \in D,           \end{cases}
\notag
\end{equation}
where $u^I$ denotes the incident wave and $\phi,\psi $ are appropriate densities (see \cite{Ammari-Kang-2}).\\
When $\epsilon_{0}$ is variable, we can represent $V$ as 
\begin{equation}
 V=\begin{cases}
U^t+S_{D}^{\kappa_0}\phi, \ \text{in } \ \mathbb{R}^{3}\backslash{ \overline{D}}, \\
S_{D}^{\kappa_1}\psi, \ \text{in } D, 
 \end{cases}
\label{variable-eps-model-single-repV}
\end{equation}
where 
\begin{eqnarray}
S_{D}^{\kappa_0}\phi&=&\int_{\partial D}G^{\kappa_0}(x,y)\phi(y) ds(y),\nonumber\\
S_{D}^{\kappa_1}\psi&=&\int_{\partial D}\Phi^{\kappa_1}(x,y)\psi(y) ds(y),\nonumber
\end{eqnarray}
and $U^t$ is the solution of the following scattering problem;
\begin{equation}
 \begin{cases}
 ( \Delta + \kappa_{0}^{2}(x)) U^t =0 \ \text{in} \ \mathbb{R}^{3}, \\
   \frac{\partial U^{s}}{\partial \vert x \vert}-i \kappa_{0} U^{s} =o (\frac{1}{\vert x \vert}) , \ \vert x \vert \rightarrow \infty \ (S.R.C),\\
   U^t(x,d)=U^s(x,d)+e^{i\kappa_o{x}\cdot{d}}.
 \end{cases}
\label{variable-eps-model-single-background}
\end{equation}
From \eqref{variable-eps-model-single-repV}, we get 
\begin{equation}
 V^\infty(\hat{x})=
U^\infty(\hat{x})+\int_{\partial D}(G^{\kappa_0})^{\infty}(\hat{x},y)\phi(y) ds(y),
\label{variable-eps-model-single-repVU-far-1}
\end{equation}
where  $(G^{\kappa_0})^{\infty}(\hat{x},y)$ is the far-field of the Green's function $G^{\kappa_0}(x,y)$ of the problem $\eqref{variable-eps-model-single-background}$. The mixed reciprocity condition implies that
$$(G^{\kappa_0})^{\infty}(\hat{x},y)= U^t(y,-\hat{x}).$$
From \eqref{variable-eps-model-single-repVU-far-1}, we have
\begin{equation}
 V^\infty(\hat{x},d)=
U^\infty(\hat{x},d)+ U^t(z,-\hat{x})\int_{\partial D}\phi+ \nabla{U^t}(z,-\hat{x})\cdot\int_{\partial D}(y-z)\phi(y) ds(y)+\mathcal{O}(\delta^3\|\phi\|)
\label{variable-eps-model-single-repVU-far-2}
\end{equation}
and from \eqref{variable-eps-model-single-repV}, we have
\begin{equation}
 V^s(x,d)=
U^s(x,d)+ G^{\kappa_0}(x,z)\int_{\partial D}\phi+ \nabla{G^{\kappa_0}}(x,z)\cdot\int_{\partial D}(y-z)\phi(y) ds(y)+\mathcal{O}(d^{-3}(x, z)\delta^3\|\phi\|)
\label{variable-eps-model-single-repV-2}
\end{equation}
where $d(x, z)$ is the Euclidean distance from $x$ to $z$.

In the following theorem, we state the asymptotic expansion of the far-fields and the near-fields in terms of $\delta$ at the first order.
\begin{theorem}\label{Main-theorem-anal}
Assume that $D$ is a Lipschitz domain, $\mu_0, \mu_1$ and $\epsilon_1$ are positive constants while $\epsilon_0$ is a Lipschitz regular function and equals $1$ outside a bounded and measurable set $\Omega$. 
In addition, the coefficients $\epsilon_0$ and $\mu_0$ are assumed to be uniformly bounded in terms of the diameter $\delta$ of $D$. 
Then, at a fixed frequency $\omega$, we have the following asymptotic expansions for the far-fields and the near-fields as $\delta \rightarrow 0$:
 \begin{enumerate} 
  \item \textbf{Far-fields:} 
  In the case of electric nanoparticles, 
  \begin{equation}\label{Elc-nano-far-1}
 \begin{aligned}
 V^{\infty}(\hat{x}, d)-U^{\infty}(\hat{x}, d) &= \omega^{2} \mu_{0} (\epsilon_{1}-\epsilon_{0}) U^{t}(z, d)U^t(z,-\hat{x}) \delta^{3} \vert B \vert \\
&\quad -\nabla_{z} U^t(z,-\hat{x}) \Big(\sum^{3}_{i=1} \partial_{i} U^{t}(z, d) \delta^{3} \int_{\partial B}y^j[\lambda Id+(K_{B}^0)^*]^{-1}(\nu_x\cdot \nabla x^i)(y)\  ds(y) \Big) +\mathcal{O}(\delta^{4-\alpha}),  
\end{aligned}
\end{equation}
where $ 0<\alpha<1$,  and in the case of magnetic nanoparticles,
\begin{equation}\label{Elc-nano-far-2}
\begin{aligned}
V^{\infty}(\hat{x}, d)-U^{\infty}(\hat{x}, d) &= \omega^{2} \mu_{0} (\epsilon_{1}-\epsilon_{0}) U^{t}(z, d)U^t(z,-\hat{x}) \delta^{3} \vert B \vert \\
&\quad -\nabla_{z} U^t(z,-\hat{x}) \Big(\sum^{3}_{i=1} \partial_{i} U^{t}(z, d) \delta^{3}  \int_{\partial B}y^j[\lambda Id+(K_{B}^0)^*]^{-1}(\nu_x\cdot \nabla x^i)(y)\  ds(y)\Big) +\mathcal{O}(\delta^{4} \vert log \ \delta \vert),  
 \end{aligned}
\end{equation}
where $0<\alpha<1 $.
  \item \textbf{Near-fields:} 
  \begin{equation}\label{Elc-nano-near-al1}
  \begin{aligned}
V^s(x, d)-U^s(x, d)&=G^{\kappa_0}(x, z)[\omega^2 \mu_0 (\epsilon_1-\epsilon_0)U^t(z,d)\delta^3 \vert B \vert+ \mathcal{O}(\delta^{4} \vert log \ \delta \vert) ]\\ 
&-\nabla G^{\kappa_0}(x, z)\cdot ( M \nabla U^t(z,d))\delta^3 +\mathcal{O}(d^{-2}(x,z) \delta^{4-\alpha})
+\mathcal{O}(d^{-3}(x,z) \delta^{4}),\ 0<\alpha<1,
\end{aligned}
\end{equation}
in the case of electric nanoparticles, and
\begin{equation}\label{Elc-nano-near}
  \begin{aligned}
V^s(x, d)-U^s(x, d)&=G^{\kappa_0}(x, z)[\omega^2 \mu_0 (\epsilon_1-\epsilon_0)U^t(z,d)\delta^3 \vert B \vert+ \mathcal{O}(\delta^{4} \vert log \ \delta \vert) ]\\ 
&-\nabla G^{\kappa_0}(x, z)\cdot ( M \nabla U^t(z,d))\delta^3 + \mathcal{O}(d^{-2}(x,z) \delta^{4+\alpha} \vert log \ \delta \vert) +\mathcal{O}(d^{-3}(x,z) \delta^{4}), \ 0<\alpha<1,
\end{aligned}
\end{equation}
in the case of magnetic nanoparticles.
\end{enumerate}
Here $M:=(M_{ij})_{i,j=1}^3, M_{ij}:=\int_{\partial B}y^j[\lambda Id+(K_{B}^0)^*]^{-1}(\nu_x\cdot \nabla x^i)(y)\  ds(y)$ is the polarization tensor and $\lambda=\frac{1}{2}\frac{\mu_0+\mu_1}{\mu_0-\mu_1} $.
\end{theorem}
%
\subsection{Application to imaging}
For simplicity of exposition, we consider $B$ to be a ball.
\begin{enumerate}
\item \textbf{Using electric nanoparticles.} 
In this case, we deduce from (\ref{Elc-nano-far-1}) that
\begin{equation}\label{inv1-1}
V^\infty(\hat{x},d)-U^\infty(\hat{x},d)=U^t(z,-\hat{x})U^t(z,d)\omega^2\mu_0\delta^{3-\alpha}|B|+\mathcal{O}(\delta^3).
\end{equation}
In particular for $\hat{x}=-d$, we get 
\begin{equation}\label{inv1-3}
V^\infty(-d,d)-U^\infty(-d,d)=\omega^2\mu_0\vert{B}\vert (U^t(z,d))^2\delta^{^{3-\alpha}}+\mathcal{O}(\delta^3).
\end{equation}
Assume that we have at hand the farfield measured before using the nanoparticles, i.e. $U^\infty(\hat{x},d)$, and after using the nanoparticles, i.e. $V^\infty(\hat{x},d)$. 
Then from (\ref{inv1-3}), we can compute ${\pm}U^t(z,d)$ with an error of the order $\mathcal{O}(\delta^\alpha)$.
By numerical differentiation we can compute ${\Delta}U^t(z)$, in the regions where $U^t(z,d)$ does not vanish, and hence 
\begin{equation}\label{comp-k_0}
\kappa^2_0(z):=-\frac{{\Delta}U^t(z,d)}{U^t(z,d)}. 
\end{equation}
Observe that, we use the backscattering data for only one incident direction $d$. 
\item \textbf{Using magnetic nanoparticles.}
In this case, we deduce from (\ref{Elc-nano-far-2}) that
\begin{equation}\label{inv2-1}
\begin{aligned}
V^\infty(\hat{x},d)-U^\infty(\hat{x},d)&=\omega^2(\epsilon_{1}-\epsilon_0)\mu_0U^t(z,-\hat{x})U^t(z,d)\delta^{3}|B| \\
&\qquad +m_0 \delta^{3}\nabla U^t(z,-\hat{x})\cdot{|B|}{\nabla}U^t(z,d)+\mathcal{O}(\delta^{3+\alpha}+\delta^{4} \vert log \ \delta \vert) \\
&=\omega^2(\epsilon_{1}-\epsilon_0)\mu_0U^t(z,-\hat{x})U^t(z,d)\delta^{3}|B|+m_0 \delta^{3}\nabla U^t(z,-\hat{x})\cdot{|B|}{\nabla}U^t(z,d)+\mathcal{O}(\delta^{3+\alpha}),
\end{aligned}
\end{equation}
where $m_0$ is related to the polarization tensor $M$ as $M=m_0 \vert B\vert Id +\mathcal{O}(\delta^{\alpha})$. Following \cite{Ammari-Kang-1} (Page $83$) and using the fact that $ \lambda= \frac{1}{2} \frac{\mu_{0}+\mu_{1}}{\mu_{0}-\mu_{1}}$, we can derive that $m_{0}=-\frac{3}{2} $ \footnote{From \cite{Ammari-Kang-1}, it follows that $\lambda= \frac{k+1}{2(k-1)} \Rightarrow  2 \lambda-1= \frac{k+1}{k-1}-1=\frac{2}{k-1} \quad
\Rightarrow k= \frac{2}{2\lambda-1} +1=\frac{2\lambda+1}{2 \lambda-1}. $ Therefore $k-1=\frac{2}{2\lambda-1},\ k+2= \frac{2\lambda+1}{2 \lambda-1}+2= \frac{6 \lambda-1}{2\lambda-1}. $ Also $\lambda-\frac{1}{6}=\frac{\mu_{0}+2\mu_{1}}{3(\mu_{0}-\mu_{1})}. $ Therefore $M:=\frac{3(k-1)}{k+2} \vert B \vert Id =\frac{6}{6\lambda-1} \vert B \vert Id= \frac{1}{\lambda-\frac{1}{6}} \vert B \vert Id =\frac{3(\mu_{0}-\mu_{1})}{\mu_{0}+2\mu_{1}} \vert B \vert Id. $ Now we can write $\frac{\mu_{0}-\mu_{1}}{\mu_{0}+2\mu_{1}}= \frac{\mu_{0}+2\mu_{1}}{\mu_{0}+2\mu_{1}}-\frac{3\mu_{1}}{\mu_{0}+2\mu_{1}}=1-\frac{3}{\frac{\mu_{0}}{\mu_{1}}+2}=-\frac{1}{2}+\mathcal{O}(\delta^{\alpha}) $ which gives us the required expansion and the value of $m_{0}$.}.\\
From (\ref{Elc-nano-near}) we deduce that
\begin{equation}\label{inv2-2}
\begin{aligned}
V^s(x, d)-U^s(x, d)&=G^{\kappa_0}(x, z)[\omega^2 \mu_0 (\epsilon_1-\epsilon_0)U^t(z,d)\delta^3 \vert B \vert +\mathcal{O}(\delta^4 \vert log \ \delta \vert)]\\ 
&\quad -\nabla G^{\kappa_0}(x, z)\cdot ( M \nabla U^t(z,d))\delta^3 + \mathcal{O}(d^{-2}(x,z) \delta^{4+\alpha} \vert log \ \delta \vert) 
+\mathcal{O}(d^{-3}(x,z) \delta^{4}) \\
&=G^{\kappa_0}(x, z)\Big[ \omega^2 \mu_0 (\epsilon_{1}-\epsilon_0) U^t(z,d)\delta^3 \vert B \vert+\mathcal{O}(\delta^4 \vert log\ \delta \vert) \Big] \\
& -\nabla G^{\kappa_0}(x, z) m_{0} \nabla U^t(z,d) \vert B \vert \delta^3+ \nabla G^{\kappa_0}(x, z)  \nabla U^t(z,d) \delta^{3+\alpha} 
 + \mathcal{O}(d^{-2}(x,z) \delta^{4+\alpha} \vert log \ \delta \vert)\\
 &\qquad +\mathcal{O}(d^{-3}(x,z) \delta^{4})\\
&=\omega^{2} \mu_{0} (\epsilon_{1}-\epsilon_{0}) U^{t}(z,d) \vert B \vert \frac{\delta^{3}}{4 \pi d(x,z)}+\mathcal{O}(\delta^{3}) +\mathcal{O}(\delta^{3} d(x,z)) \\
&\quad +\mathcal{O}(d^{-1}(x,z)\delta^4 \vert log\ \delta \vert)+\mathcal{O}(\delta^4 \vert log\ \delta \vert)+\mathcal{O}(d(x,z) \delta^4 \vert log\ \delta \vert)\\
&\quad -m_{0} \nabla U^{t}(z,d) \vert B \vert \frac{(x-z) \delta^{3}}{4 \pi d^{3}(x,z)}+ \mathcal{O}(\delta^{3})\\
&\quad +\nabla U^{t}(z,d) \frac{(x-z) \delta^{3+\alpha}}{4 \pi d^{3}(x,z)}+\mathcal{O}(\delta^{3+\alpha})
 + \mathcal{O}(d^{-2}(x,z) \delta^{4+\alpha} \vert log \ \delta \vert) +\mathcal{O}(d^{-3}(x,z) \delta^{4}),
\end{aligned}
\end{equation}
using the facts that when $x$ is close to $z$, see \cite{S-1},
\begin{equation}
\begin{aligned}
G^{k_{0}}(x,z)&=\frac{1}{4 \pi d(x,z)}+\frac{i}{4 \pi} k_{0}(z)+\mathcal{O}(d(x,z)), \\
\nabla G^{k_{0}}(x,z)&=\frac{x-z}{4 \pi d^{3}(x,z)}+\mathcal{O}(1) .
\end{aligned}
\notag
\end{equation}
Assume that we can measure the far-fields and the near fields before as well as after using the nanoparticles. 
\begin{enumerate}
\item 
From \eqref{inv2-1}, we can reconstruct 
$\omega^{2} \vert B \vert (\epsilon_{1}-\epsilon_0) \mu_0\left(U^t(z,d)\right)^2+m_0\vert B\vert\nabla{U}^t(z,d)\cdot\nabla{U}^t(z,d)$ with an error of the order $\delta^{\alpha}$ using one direction of incidence $d$ and the corresponding backscattering 
direction $\hat{x}=-d$.
 \item Multiplying \eqref{inv2-2} by $\frac{d^2(x, z)}{\delta^3}$ we get
 \begin{equation}\label{inv2-3}
 \begin{aligned}
\frac{d^2(x, z)}{\delta^3}[V^s(x,d)-U^s(x,d)]&=-m_{0} \nabla U^{t}(z,d) \vert B \vert \cdot \frac{(x-z)}{4\pi d(x,z)}+ \nabla U^{t}(z,d)\cdot \frac{(x-z)}{4 \pi d(x,z)} \delta^{\alpha} \\
&\quad +\omega^{2} \mu_{0} (\epsilon_{1}-\epsilon_{0}) U^{t}(z,d) \frac{\vert B \vert}{4 \pi} d(x,z)  \\
&\quad +\mathcal{O}(d^{2}(x,z)) +\mathcal{O}(\delta^{\alpha} d^{2}(x,z)) + \mathcal{O}(d^{3}(x,z))\\ 
&\quad +\mathcal{O}(\delta^{1+ \alpha} \vert log \ \delta \vert)+\mathcal{O}(\frac{\delta}{d(x,z)})+\mathcal{O}(d(x,z) \delta \vert log\ \delta\vert).
\end{aligned}
\end{equation}
Thus from \eqref{inv2-2}, we can reconstruct $\nabla{U}^t(z,d)$ with an error of the order $\delta^{\min\{\alpha,1-\alpha\}}$ if $d(x, z) \sim \mathcal{O}(\delta^\alpha), 0<\alpha<1$.
\end{enumerate}
Hence, we can reconstruct $(\epsilon_{1}-\epsilon_0)\left(U^t(z,d)\right)^2$ with an error of the order $\delta^{\min\{\alpha,1-\alpha\}}$ if we use the far-field data and the near fields collected at a distance, from the tumor location, of the order 
$d(x, z) \sim \mathcal{O}(\delta^\alpha)$.\\
Let now $K$ be a smooth region where we have sent the nanoparticles. Knowing $\nabla U^t(z,d)$, for a sample of points $z \in K$, we can recover $-\omega^2 \mu_0\epsilon_0(z)U^t(z,d)$ (and hence $\epsilon_0U^t(z,d)$) by taking the divergence of $\nabla U^t(z,d)$  as $\nabla \cdot \nabla U^t(z,d) =-\omega^2 \mu_0\epsilon_0(z)U^t(z,d)$.\\
Combining $\epsilon_{0} U^t(z,d)$ and $(\epsilon_{1}-\epsilon_0)\left(U^t(z,d)\right)^2$ we can recover $\epsilon_0(z)$.
For instance, in the case $\epsilon_1>\epsilon_0(z),\ z\in \Omega $, we can consider the second-order polynomial equation $ A\epsilon_{0}^{2}-(\epsilon_1-\epsilon_0)B^2=0$ where $A=(\epsilon_1-\epsilon_0) (U^t(z,d))^2 $ and $B=\epsilon_0 U^{t}(z,d) $. It is easy to see that this equation has a positive root $\epsilon_0 $ since $\epsilon_1>\epsilon_0 $.

\end{enumerate}

In both cases, the error (which is $\mathcal{O}(\delta^\alpha)$, where $\alpha \in (0, 1)$, using electric nanoparticles, or $\mathcal{O}(\delta^{\min\{\alpha,1-\alpha\}})$, for $\alpha \in (0,1)$,
using magnetic nanoparticles) propagates through the numerical differentiation which is a moderately unstable step, i.e. with a loss of a polynomial rate of the error.

\begin{remark}\label{Remark-on-measured-data}

We observe that to get the reconstruction of the coefficients $\epsilon_0$
\begin{enumerate}
 \item using electric nanoparticles, we need to measure the far-fields before and after the injection of these particles.
However, we only need one single backscattered direction $d \in \mathbb{S}^2$.

\item using magnetic nanoparticles, we need to measure the far-fields before and after the injection of these particles in one single backscattered direction $d\in \mathbb{S}^2$ and
the near-fields in three different points $x_j, j=1, 2, 3$ such that the unit vectors $\frac{(x_j-z)}{d(x_j, z)}$ are linearly independent.
\end{enumerate}
In both cases, the error of the reconstruction is fixed by the parameter $\alpha$ which is related to the kind of (relative permittivity or permeability of the)  nanoparticles used. 
The distance from the place where to collect the near-field data to the tumor (which is of the order $\delta^{\alpha}$ for $\alpha \in (0,1)$) is also fixed by the kind of magnetic nanoparticles used.
 
\end{remark}
\section{Proof of Theorem \ref{Main-theorem-anal}}
We observe that in the scattering problem \eqref{model4}, the actual parameters modeling the nanoparticles are $\mu_1 $ and $\kappa_{1}^{2}(:= \omega^{2}  \epsilon_1 \mu_1) $. 
To handle in one shot both the electric and magnetic nanoparticles, we take $\mu_{1} \sim \delta^{-\beta},\ \beta \geq 0$ and $\kappa_{1}^{2} \sim \delta^{-\alpha} $. Then, taking $\beta=0 $ and $\alpha>0 $ means that we deal with electric nanoparticles. The case $\beta=\alpha $ covers the case of magnetic nanoparticles. 
\subsection{The case when $\epsilon_0$ is a constant}
First, we recall that in the case where $\epsilon_0$ is a constant every where in $\mathbb{R}^{3}$, we use the notation $u$ instead of $V$ and that the solution $u$ can be represented as
\begin{equation}
u(x)=\begin{cases}
      u^{I}+ S^{\kappa_{0}}_{D} \phi(x), \ x \in \mathbb{R}^{3}\diagdown \overline D,   \\
       S^{\kappa_{1}}_{D} \psi(x) , \ x \in D.           \end{cases}
\label{sol1}
\end{equation}
By the jump relations, the densities $\phi$ and $\psi$ are solutions of the system of integral equations
\begin{equation}
\begin{aligned}
& S^{\kappa_{1}}_{D} \psi- S^{\kappa_{0}}_{D} \phi = u^{I},\\
& \frac{1}{\mu_{1}} \Big[\frac{1}{2} Id+(K^{\kappa_{1}}_{D})^{*}\Big] \psi -\frac{1}{\mu_{0}} \Big[-\frac{1}{2} Id + (K^{\kappa_{0}}_{D})^{*} \Big]\phi =\frac{1}{\mu_{0}} \frac{\partial u^{I}}{\partial \nu}.
\end{aligned}
 \label{sing-obs-1}
\end{equation}
%
%
%
%
Using the fact that $[-\frac{1}{2} Id+(K^{\kappa_{0}}_{D})^{*}]$ is invertible, from the second equation of \eqref{sing-obs-1}, we have 
\begin{equation}
\phi=\frac{\mu_{0}}{\mu_{1}} \Big[-\frac{1}{2} Id+(K^{\kappa_{0}}_{D})^{*} \Big]^{-1} \Big[\frac{1}{2} Id+(K^{\kappa_{1}}_{D})^{*} \Big] \psi-\Big[-\frac{1}{2} Id+(K^{\kappa_{0}}_{D})^{*} \Big]^{-1} \frac{\partial u^{I}}{\partial \nu}.
\label{sing-obs-2}
\end{equation}
From the first identity, we have 
\begin{equation}
 \psi=(S^{\kappa_{1}}_{D})^{-1} u^{I}+(S^{\kappa_{1}}_{D})^{-1} S^{\kappa_{0}}_{D} \phi.
\label{sing-obs-3}
\end{equation}
Substituting this in \eqref{sing-obs-2}, we obtain
\begin{equation}
\begin{aligned}
\phi&=\frac{\mu_{0}}{\mu_{1}} \Big[-\frac{1}{2} Id+(K^{\kappa_{0}}_{D})^{*} \Big]^{-1} \Big[\frac{1}{2} Id+(K^{\kappa_{1}}_{D})^{*} \Big] (S^{\kappa_{1}}_{D})^{-1} u^{I} \\
&\ \ +\frac{\mu_{0}}{\mu_{1}} \Big[-\frac{1}{2} Id+ (K^{\kappa_{0}}_{D})^{*} \Big]^{-1} \Big[\frac{1}{2} Id+(K^{\kappa_{1}}_{D})^{*} \Big] (S^{\kappa_{1}}_{D})^{-1} S^{\kappa_{0}}_{D} \phi \\
&\ \ -\Big[-\frac{1}{2} Id+(K^{\kappa_{0}}_{D})^{*} \Big]^{-1} \frac{\partial u^{I}}{\partial \nu}, 
\end{aligned}
\notag
\end{equation}
which implies
\begin{equation}
\begin{aligned}
(Id-A S^{\kappa_{0}}_{D}) \phi= A u^{I}- [-\frac{1}{2} Id+(K^{\kappa_{0}}_{D})^{*}]^{-1} \frac{\partial u^{I}}{\partial \nu},
\end{aligned}
\label{sing-obs-4}
\end{equation}
where $A=\frac{\mu_{0}}{\mu_{1}} \Big[-\frac{1}{2} Id+(K^{\kappa_{0}}_{D})^{*} \Big]^{-1} \Big[\frac{1}{2} Id+(K^{\kappa_{1}}_{D})^{*} \Big] (S^{\kappa_{1}}_{D})^{-1} $. Therefore
\begin{equation}
\phi= \Big[Id-A S^{\kappa_{0}}_{D} \Big]^{-1} A u^{I}- \Big[Id-A S^{\kappa_{0}}_{D} \Big]^{-1} \Big[-\frac{1}{2} Id+ (K^{\kappa_{0}}_{D})^{*} \Big]^{-1} \frac{\partial u^{I}}{\partial \nu},  
\label{sing-obs-5}
\end{equation}
provided $\Big[Id-A S^{\kappa_{0}}_{D}\Big] $ is an invertible operator.\\
\textbf{Notation:} For $\phi \in L^{2}(\partial D), \psi \in L^{2}(\partial B) $, we shall denote by $\hat{\phi}, \check{\psi} $ the functions
\[\begin{aligned}
\hat{\phi}(\xi)&:= \phi(\delta \xi+z), \ \xi \in \partial B, \\
\check{\psi}(x)&:= \psi(\frac{x-z}{\delta}), \ x \in \partial D.
\end{aligned}
 \]
%
\begin{lemma}\label{estimate-in-lambda}
 Let $\kappa_0>0$. The operator $\lambda\; Id +(K^{\kappa_0}_D)^{*}: L^2(\partial D)\longrightarrow L^2(\partial D)$ is invertible with 

\begin{equation} \label{k-0=0-estimate-in-lambda}
 \Vert (\lambda\; Id +(K^{\kappa_0}_D)^{*})^{-1}\Vert_{\mathcal{L}(L^2(\partial D), L^2(\partial D))} \sim (\lambda -\frac{1}{2})^{-1},
\end{equation}
if $\mu_1 \sim \delta^{\beta},\; \beta <2$.

\end{lemma}

\begin{proof}
Let us first assume that $\mu_1 \sim \delta^{-\beta}, \beta > 0$. Then $\lambda\; Id +(K^{0}_D)^{*}$ is invertible and 
$(\lambda\; Id +(K^{0}_D)^{*})^{-1}$ is bounded, for $\delta>0$ small enough. To see this, we write 
$$
(\lambda\; Id +(K^{0}_D)^{*})^{-1}=(-\frac{1}{2} Id +(K^{0}_D)^{*})^{-1}(Id +\frac{\mu_0}{\mu_0-\mu_1}(-\frac{1}{2} Id +(K^{0}_D)^{*})^{-1})^{-1}
$$
where $\frac{\mu_0}{\mu_0-\mu_1}$ is small as $\delta>0$ is small. 
\bigskip

For the case $\mu_1 \sim \delta^{\beta}, \beta \geq 0$, we use the decomposition $L^2(\partial D)=1 \oplus L_0^2(\partial D)$ and recall that $(\frac{1}{2} Id + K^{0}_D)\mid_{1}\equiv 0$, $(\lambda\; Id + K^{0}_D)\mid_{1}\equiv \lambda -\frac{1}{2}$ (since $K^{0}_{D}(1)=-\frac{1}{2} $), while $(\lambda\; Id + K^{0}_D)\mid_{L_0^2(\partial D)}$ is an isomorphism with a bounded inverse in terms of $\lambda$, 
$\vert \lambda \vert \geq \frac{1}{2}$, see \cite{Ammari-Kang-1}. Hence 
\begin{equation} \label{k-0=0-estimate-in-lambda-pr}
 \Vert (\lambda\; Id +(K^{0}_D)^{*})^{-1}\Vert_{\mathcal{L}(L^2(\partial D), L^2(\partial D))}=\Vert (\lambda\; Id + K^{0}_D)^{-1}\Vert_{\mathcal{L}(L^2(\partial D), L^2(\partial D))} \sim (\lambda -\frac{1}{2})^{-1}.
\end{equation}
Combining the estimates from the above two cases, we can now prove the desired result. To do so, we first write 
 $$
 \lambda\; Id +(K^{\kappa_0}_D)^{*} =[\lambda\; Id +(K^{0}_D)^{*}][Id +(\lambda\; Id +(K^{0}_D)^{*})^{-1}[(K^{\kappa_0}_D)^{*}-(K^{0}_D)^{*}]].
 $$
 As $\Vert (K^{\kappa_0}_D)^{*}-(K^{0}_D)^{*} \Vert_{\mathcal{L}(L^{2}(\partial D),L^{2}(\partial D))} \sim \delta^{2}$ (see \cite{Challa-Sini-1, Challa-Sini-2}) and 
 $\Vert (\lambda\; Id +(K^{0}_D)^{*})^{-1}\Vert_{\mathcal{L}(L^2(\partial D), L^2(\partial D))} \sim (\lambda -\frac{1}{2})^{-1}\sim \delta^{-\beta}$, then 
 $\Vert (\lambda\; Id +(K^{0}_D)^{*})^{-1}[(K^{\kappa_0}_D)^{*}-(K^{0}_D)^{*}]\Vert_{\mathcal{L}(L^{2}(\partial D),L^{2}(\partial D))} \sim \delta^{2-\beta} << 1,$ if $\beta<2$ \footnote{In fact we can also take $\beta=2$ under the assumptions implying that
$ \Vert (\lambda\; Id +(K^{\kappa_0}_D)^{*})^{-1}\Vert\; \Vert (K^{\kappa_0}_D)^{*}-(K^{0}_D)^{*}\Vert <1$. These assumptions on $\mu_0, \mu_1/\delta^2$ and $\partial D$ (through its Lipschitz character) 
can be derived by expanding 
$\Phi^{\kappa_0}-\Phi^0$.}.
\end{proof}

\begin{remark}
In case $\kappa_{0}$ is not constant, we shall see later that $\Vert (K^{\kappa_0}_D)^{*}-(K^{0}_D)^{*} \Vert_{\mathcal{L}(L^{2}(\partial D),L^{2}(\partial D))} \sim \delta^{2} (log \ \delta) $ and not of order $\delta^{2} $. Therefore in the above lemma, we have to choose $\alpha $ accordingly. But surely this holds if $\beta <1 $.
\end{remark}
\begin{lemma}\label{scaling}
 The operator $\Big[Id-A S^{\kappa_{0}}_{D}\Big] $ is invertible and $\Vert [Id-A S^{\kappa_{0}}_{D} ]^{-1} \Vert_{\mathcal{L}(H^{1}(\partial D),L^{2}(\partial D))} $ is uniformly bounded if $\mu_1\sim \delta^{-\beta}, \; \beta \geq 0$, and $\kappa_{1}^{2} \sim \delta^{-\alpha},\ 0\leq \alpha<1 $.
\end{lemma}
\begin{proof}
Following \cite{Challa-Sini-1, Challa-Sini-2}, we can prove that the operator $[-\frac{1}{2} Id+ (K^{\kappa_{0}}_{D})^{*}]^{-1}$ is uniformly bounded and the operator $[\frac{1}{2} Id+(K^{\kappa_{1}}_{D})^{*}]$ is uniformly bounded provided $\alpha<2 $. We next prove that $(S^{\kappa_{1}}_{D})^{-1} S^{\kappa_{0}}_{D} $ is uniformly bounded with respect to $\delta$ as well. To see this, following \cite{Challa-Sini-1, Challa-Sini-2}, let us consider 
\begin{eqnarray}\label{expansdsdinv}
 \left(S_{D}^{\kappa_1}\right)^{-1} S_{D}^{\kappa_0}
 &=& \left(S_{D}^{0}+S_{D}^{d_{\kappa_1}}\right)^{-1}\left(S_{D}^{0}+S_{D}^{d_{\kappa_0}}\right)\nonumber\\
&=& \left(Id+\left(S_{D}^{0}\right)^{-1}S_{D}^{d_{\kappa_1}}\right)^{-1}\left(S_{D}^{0}\right)^{-1}\left(S_{D}^{0}+S_{D}^{d_{\kappa_0}}\right)\nonumber\\ 
&=& \left(Id+\left(S_{D}^{0}\right)^{-1}S_{D}^{d_{\kappa_1}}\right)^{-1}\left(Id+\left(S_{D}^{0}\right)^{-1}S_{D}^{d_{\kappa_0}}\right)\nonumber\\ 
&=& \left(Id+\sum_{l=1}^{\infty} (-1)^{l} \left(\left(S_{D}^{0}\right)^{-1}S_{D}^{d_{\kappa_1}}\right)^{l}\right)\left(Id+\left(S_{D}^{0}\right)^{-1}S_{D}^{d_{\kappa_0}}\right)\nonumber\\
&=& Id+\sum_{l=1}^{\infty} (-1)^{l} \left(\left(S_{D}^{0}\right)^{-1}S_{D}^{d_{\kappa_1}}\right)^{l}\left(Id+\left(S_{D}^{0}\right)^{-1}S_{D}^{d_{\kappa_0}}\right)+\left(S_{D}^{0}\right)^{-1}S_{D}^{d_{\kappa_0}},
 \end{eqnarray}
 where $S^{0}_{D} \psi(x)= \int_{\partial D} \frac{1}{4 \pi \vert x-y \vert} \psi(y) \ ds(y) $ and $S^{d_{\kappa_{j}}}_{D} \psi(x) = \int_{\partial D} \frac{e^{i \kappa_{j} \vert x-y \vert}-1}{4 \pi \vert x-y \vert} \psi(y) \ ds(y), j=0,1 $.\\
 The equality \eqref{expansdsdinv} holds for $\left\|\left(S_{D}^{0}\right)^{-1}S_{D}^{d_{\kappa_1}}
 \right\|_{\mathcal{L}\left(L^2(\partial D), L^2(\partial D) \right)}<1$. 
 From \cite{Challa-Sini-2}, this condition holds for $\delta$ such that $\underbrace{\left(\frac{|\partial B|}{2\pi}\left\|{S^{0}_{B}}^{-1}\right\|_{\mathcal{L}\left(H^1(\partial B), L^2(\partial B) \right)}(1+ \kappa_{1})\kappa_{1}\right)}_{=:C}\, \delta<1$. We note that this also necessitates the condition $\alpha<1 $.
 For $j=0,1 $, we also have
 \begin{equation}\label{neumannc}
 \left\|\left(S_{D}^{0}\right)^{-1}S_{D}^{d_{\kappa_j}}\right\|_{\mathcal{L}\left(L^2(\partial D), L^2(\partial D) \right)}\leq\underbrace{\frac{|\partial B|}{2\pi}\left\|{S^{0}_{B}}^{-1}\right\|_{\mathcal{L}\left(H^1(\partial B), L^2(\partial B) \right)}(1+ \kappa_j)\kappa_j}_{=:C_j}\, \delta ,
 \end{equation}
where $C_1=\mathcal{O}(\delta^{-\alpha}) $.
Observe that,
\begin{equation}
\begin{aligned}
&\left\|\sum_{l=1}^{\infty} (-1)^{l} \left(\left(S_{D}^{0}\right)^{-1}S_{D}^{d_{\kappa_1}}\right)^{l}\left(Id+\left(S_{D}^{0}\right)^{-1}S_{D}^{d_{\kappa_0}}\right)+\left(S_{D}^{0}\right)^{-1}S_{D}^{d_{\kappa_0}}\right\|_{\mathcal{L}\left(L^2(\partial D), L^2(\partial D) \right)}\nonumber\\
&\leq \left\|\sum_{l=1}^{\infty} (-1)^{l} \left(\left(S_{D}^{0}\right)^{-1}S_{D}^{d_{\kappa_1}}\right)^{l}\right\|_{\mathcal{L}\left(L^2(\partial D), L^2(\partial D) \right)}\left\|\left(Id+\left(S_{D}^{0}\right)^{-1}S_{D}^{d_{\kappa_0}}\right)\right\|_{\mathcal{L}\left(L^2(\partial D), L^2(\partial D) \right)}\\ 
&\qquad +\left\|\left(S_{D}^{0}\right)^{-1}S_{D}^{d_{\kappa_0}}\right\|_{\mathcal{L}\left(L^2(\partial D), L^2(\partial D) \right)}\nonumber\\
&\leq \left(\frac{C_1}{1-C_1\delta}(1+C_0\delta)+C_0\right)\delta.
\end{aligned}
\notag
\end{equation}
\begin{equation}\label{single}
\begin{aligned}
&\text{Therefore}\ (S^{\kappa_{1}}_{D})^{-1} S^{\kappa_{0}}_{D}\  \text{is uniformly bounded, since}\ \alpha<1 .
\end{aligned}
\end{equation} 
Using \eqref{expansdsdinv}-\eqref{neumannc}, we can now write
\begin{equation}
Id- A S^{\kappa_{0}}_{D}= Id-\frac{\mu_{0}}{\mu_{1}} [-\frac{1}{2} Id+(K^{\kappa_{0}}_{D})^{*} ]^{-1} [\frac{1}{2} Id+(K^{\kappa_{1}}_{D})^{*}] +\mathcal{O}(\frac{\mu_0}{\mu_1}\delta^{1-\alpha}).
\notag
\end{equation}
   %
Multiplying both sides by $[-\frac{1}{2} Id+(K^{\kappa_{0}}_{D})^{*}]$, we then have
\begin{equation}
\begin{aligned}
&[-\frac{1}{2} Id+(K^{\kappa_{0}}_{D})^{*}] [Id-A S^{\kappa_{0}}_{D} ] = [-\frac{1}{2} Id+(K^{\kappa_{0}}_{D})^{*}]-\frac{\mu_{0}}{\mu_{1}} [\frac{1}{2} Id+(K^{\kappa_{1}}_{D})^{*}] + [-\frac{1}{2} Id+ (K^{\kappa_{0}}_{D})^{*}] \mathcal{O}(\frac{\mu_0}{\mu_1}\delta^{1-\alpha}) \\
&= [-\frac{1}{2} Id+(K^{0}_{D})^{*}] -\frac{\mu_{0}}{\mu_{1}} [\frac{1}{2} Id+(K^{0}_{D})^{*}] + [(K^{\kappa_{0}}_{D})^{*}-(K^{0}_{D})^{*}] -\frac{\mu_{0}}{\mu_{1}} [(K^{\kappa_{1}}_{D})^{*}-(K^{0}_{D})^{*}]+ [-\frac{1}{2} Id+(K^{\kappa_{0}}_{D})^{*}] \mathcal{O}(\frac{\mu_0}{\mu_1}\delta^{1-\alpha})\\
&=(1-\frac{\mu_{0}}{\mu_{1}}) (\lambda Id+(K^{0}_{D})^{*})+ \mathcal{O}(\frac{\mu_0}{\mu_1} \delta^{1-\alpha})+\mathcal{O}(\frac{\mu_0}{\mu_1}\delta^{2-\alpha})+\mathcal{O}(\delta^2),
\end{aligned}
\notag
\end{equation}
since $[(K^{\kappa_{0}}_{D})^{*}-(K^{0}_{D})^{*}]$ is of order $\delta^{2}$ and $[(K^{\kappa_{1}}_{D})^{*}-(K^{0}_{D})^{*}]$ is of order $\kappa_{1}^{2}\delta^{2}(=\delta^{2-\alpha}) $. Therefore if $\alpha < 1 $, the error above is atleast of the order $\delta^{1-\alpha} $.\\
Now since $\mu_{0} \neq \mu_{1}$, it follows that $[\lambda Id+ (K^{0}_{D})^{*}]$ is invertible and since norm of $[-\frac{1}{2} Id+(K^{\kappa_{0}}_{D})^{*}]$ is bounded by a constant, it then follows that $[Id-A S^{\kappa_{0}}_{D}]$ is an invertible operator and the norm of the inverse is uniformly bounded.

%
%
\end{proof}
Let $$B:=\Big[\frac{1}{2} Id+(K^{\kappa_{1}}_{D})^{*} \Big] (S^{\kappa_{1}}_{D})^{-1} u^{I} = \left.\frac{\partial S^{\kappa_{1}}_{D} }{\partial \nu}\right\vert_{-}  ((S^{\kappa_{1}}_{D})^{-1} u^{I}).$$ Then
from \eqref{sing-obs-5} it follows that
\begin{equation}
\phi= \Big[Id- A S^{\kappa_{0}}_{D} \Big]^{-1} \Big[-\frac{1}{2} Id+(K^{\kappa_{0}}_{D})^{*} \Big]^{-1} \Big[\frac{\mu_{0}}{\mu_{1}} B-\frac{\partial u^{I}}{\partial \nu} \Big]. 
\notag
\end{equation}
Let us set $W_{\kappa_{1}}:=S^{\kappa_{1}}_{D} ((S^{\kappa_{1}}_{D})^{-1} u^{I}) $. Then $W_{\kappa_{1}}$ satisfies 
\begin{equation}
 \begin{aligned}
  (\Delta+\kappa_{1}^{2}) W_{\kappa_{1}} &=0 \ \text{in}\ D \\
  W_{\kappa_{1}} &= u^{I} \ \text{on} \ \partial D.
 \end{aligned}
\notag
\end{equation}
\begin{lemma}\label{H2propWk}
We have the following behavior of $W_{\kappa_{1}}$:
\begin{eqnarray} 
 \hat{W}_{\kappa_{1}}=\hat{u}^{I} + \mathcal{O}(\delta^{2-\alpha}) \ \text{in}\ H^{2}(B)\label{H2propWk-B}\\
 W_{\kappa_{1}}=u^{I} + \mathcal{O}(\delta^{\frac{3}{2}-\alpha}) \ \text{in}\ H^{2}(D)\label{H2propWk-D}.
  \end{eqnarray}
\end{lemma}
\begin{proof}
We note that $\textcolor{black}{H:=W_{\kappa_{1}}-u^{I} }$ satisfies the elliptic equation
\begin{equation}\label{H2propWk-f-eq}
\left\{\begin{array}{cccc}
 (\Delta+ \kappa_{1}^{2}) H&= &h(\textcolor{black}{:=(\kappa_{0}^2-\kappa_{1}^2)e^{i \kappa_{0} x.d}}) &\text{in}\ D \\
 H&=& 0 & \text{on} \ \partial D.
\end{array}
\right.
\end{equation}

Multiplying the first equation of \eqref{H2propWk-f-eq} by $\bar{H}$ and by applying the integration by parts formula, we obtain:
\begin{equation}\label{H2propWk-f-eq-green}
\int_D\vert\nabla{H}\vert^2\,dx-\kappa_1^2\int_D\vert{H}\vert^2\,dx\,=\,-\int_Dh\bar{H}\,dx .
\end{equation}
Applying the Poincare inequality for the first term on L.H.S, we obtain
\begin{equation}\label{H2propWk-f-eq-poin}
C(D)\int_D\vert{H}\vert^2\,dx-\kappa_1^2\int_D\vert{H}\vert^2\,dx\,\leq\,\int_D\vert\nabla{H}\vert^2\,dx-\kappa_1^2\int_D\vert{H}\vert^2\,dx\,=\,-\int_Dh\bar{H}\,dx .
\end{equation}
Now since the first Dirichlet eigenvalue is the same as the sharp Poincare constant C(D), by the Faber-Krahn inequality, we have
\begin{equation}\label{Feber}
C(D)=\lambda_1(D)\geq\lambda_1(D^*)=\frac{C_1(B)}{\delta^2},
\end{equation}
where $D^{*} $ denotes the ball of same volume as $D $, 
$$
\lambda_1(D^*)
=\left(\frac{\omega_N}{\vert D \vert}\right)^\frac{2}{N}j_{\frac{N-2}{2},1}^2
=\frac{\left(\frac{4}{3}\pi\right)^\frac{2}{3}}{\delta^2(diam(B))^2}j_{\frac{1}{2},1}^2=:\frac{C_1(B)}{\delta^2},
$$
and $j_{p,1} $ denotes the first positive zero of the Bessel function (of first kind) $J_p $ .\\
Using the inequality \eqref{Feber} in  \eqref{H2propWk-f-eq-poin} and by applying the Cauchy Schwartz inequality, we obtain:
\begin{eqnarray}\label{H2propWk-f-eq-Cauchy}
[C_1(B)\delta^{-2}-\kappa_1^2]\,\|H\|_{L^2(D)}&\leq&\|h\|_{L^2(D)}\nonumber\\
\implies\,\|\hat{H}\|_{L^2(B)}&\leq&\frac{\delta^{2}}{[C_1(B)-\delta^{2}\kappa_1^2]}\|\hat{h}\|_{L^2(B)}.
\end{eqnarray}

Now, rewrite \eqref{H2propWk-f-eq} as,
\begin{equation}\label{H2propWk-f-eq-1}
\left\{\begin{array}{cccc}
 \Delta{ H}&= &-\kappa_{1}^{2} H+h &\text{in}\ D \\
 H&=& 0 & \text{on} \ \partial D.
\end{array}
\right.
\end{equation}
By scaling we can rewrite \eqref{H2propWk-f-eq-1} on the reference body $B$ as,
\begin{equation}\label{H2propWk-f-eq-B}
\left\{\begin{array}{cccc}
 \Delta{ \hat{H}}&= &\delta^2[-\kappa_{1}^{2}\hat{H}+\hat{h}] &\text{in}\ B \\
 \hat{H}&=& 0 & \text{on} \ \partial B.
\end{array}
\right.
\end{equation}
By applying the elliptic regularity theorem, there exists a positive constant $C_2(B)$ such that
\begin{eqnarray}\label{H2propWk-f-eq-Elliptic-B}
\|\hat{H}\|_{H^2(B)}&\leq&C_2(B)\,\|\delta^{2}(-\kappa_1^2\hat{H}+\hat{h})\|_{L^2(B)}\nonumber\\
&\leq&\delta^{2}\,C_2(B)\,[\kappa_1^2\|\hat{H}\|_{L^2(B)}+\|\hat{h}\|_{L^2(B)}]\nonumber\\
&\substack{\leq\\ \eqref{H2propWk-f-eq-Cauchy}}&\delta^{2}\,C_2(B)\left[\kappa_1^2\frac{\delta^{2}}{\vert{C_1(B)-\delta^{2}\kappa_1^2}\vert}+1\right]\|\hat{h}\|_{L^2(B)}\nonumber\\
&=&\delta^{2}\,C_2(B)\left[\frac{\kappa_1^2\delta^{2}}{[C_1(B)-\delta^{2}\kappa_1^2]}+1\right]\vert{\kappa_0^2-\kappa_1^2} \vert\,\vert\partial{B}\vert^\frac{1}{2}=\mathcal{O}(\delta^{2-\alpha}).
\end{eqnarray}
Hence, we are done with proving \eqref{H2propWk-B}.
For any function $\Xi\in H^2(D),\,D\subset\mathbb{R}^N$, one can prove the following by rescaling 
\begin{eqnarray}\label{scalingdomain-g-H2}
\delta^{\frac{N}{2}}\|\hat{\Xi}\|_{H^2(B)}\,\leq\,\|{\Xi}\|_{H^2(D)}\,\leq\,\delta^{\frac{N-4}{2}}\|\hat{\Xi}\|_{H^2(B)}.
\end{eqnarray}
In our case, we have $N=3, \Xi=H$ and so we have the following inequality:
\begin{eqnarray}\label{scalingdomain-H2}
\delta^{\frac{3}{2}}\|\hat{H}\|_{H^2(B)}\,\leq\,\|{H}\|_{H^2(D)}\,\leq\,\delta^{-\frac{1}{2}}\|\hat{H}\|_{H^2(B)}.
\end{eqnarray}

By making use of the right inequality of \eqref{scalingdomain-H2} in \eqref{H2propWk-f-eq-Elliptic-B}, we get
\begin{eqnarray}\label{H2propWk-f-eq-Elliptic-D}
\|{H}\|_{H^2(D)}\,\leq\,\delta^{-\frac{1}{2}}\|\hat{H}\|_{H^2(B)}
&\leq&C_2(B)\left[\frac{\kappa_1^2\delta^{2}}{\vert{C_1(B)-\delta^{2}\kappa_1^2}\vert}+1\right]\vert{\kappa_0^2-\kappa_1^2} \vert\,\vert\partial{B}\vert^\frac{1}{2}\delta^{\frac{3}{2}}=\mathcal{O}(\delta^{\frac{3}{2}-\alpha}).
\end{eqnarray}
Hence, we are done with proving \eqref{H2propWk-D}.
\end{proof}
\begin{lemma}\label{phi-estimate}
$\Vert \phi \Vert_{L^{2}(\partial D)}= \mathcal{O}(\delta)$, if $\mu_1 \sim \delta^{-\beta}, \ \beta \geq 0, $ and $\kappa_{1}^{2} \sim \delta^{-\alpha}, \ 0\leq \alpha<1 $.
\end{lemma}
\begin{proof}
We can write
\begin{equation}
\begin{aligned}
\phi &= \Big[Id-A S^{\kappa_{0}}_{D} \Big]^{-1} \Big[-\frac{1}{2} Id+ (K^{\kappa_{0}}_{D})^{*} \Big]^{-1} \Big[(\frac{\mu_{0}}{\mu_{1}}-1) \frac{\partial u^{I}}{\partial \nu} \Big] +\Big[Id-A S^{\kappa_{0}}_{D} \Big]^{-1} \Big[-\frac{1}{2} Id+ (K^{\kappa_{0}}_{D})^{*} \Big]^{-1} \frac{\mu_{0}}{\mu_{1}} \Big[B- \frac{\partial u^{I}}{\partial \nu} \Big] \\
&=\Big[Id-A S^{\kappa_{0}}_{D} \Big]^{-1} \Big[-\frac{1}{2} Id+ (K^{\kappa_{0}}_{D})^{*} \Big]^{-1} \Big[(\frac{\mu_{0}}{\mu_{1}}-1) \frac{\partial u^{I}}{\partial \nu} \Big]\\
&\qquad +\Big[Id-A S^{\kappa_{0}}_{D} \Big]^{-1} \Big[-\frac{1}{2} Id 
+ (K^{\kappa_{0}}_{D})^{*} \Big]^{-1} \frac{\mu_{0}}{\mu_{1}} \Big[\frac{\partial W_{\kappa_{1}}}{\partial \nu}\Big\vert_{-}- \frac{\partial u^{I}}{\partial \nu} \Big],
\end{aligned}
\notag
\end{equation}
which implies
\begin{equation}
\Vert \phi \Vert_{L^{2}(\partial D)}\leq 
C \Big\Vert \frac{\partial u^{I}}{\partial \nu} \Big\Vert_{L^{2}(\partial D)}+ C \frac{\mu_0}{\mu_1} \Big\Vert \frac{\partial W_{\kappa_{1}}}{\partial \nu}\Big\vert_{-}-\frac{\partial u^{I}}{\partial \nu} \Big\Vert_{L^{2}(\partial D)} \leq C \Big(\delta+ \frac{\mu_0}{\mu_1} \Vert \frac{\partial H}{\partial \nu} \Vert_{L^{2}(\partial D)}\Big).
\label{estim-phi-1}
\end{equation}
As $\nabla_{\xi} \hat{H} = \delta \nabla_{x} H$, then
\begin{equation}
\begin{aligned}
\Big\Vert \frac{\partial H}{\partial \nu} \Big\Vert_{L^{2}(\partial D)}&=\Big(\int_{\partial D} \Big\vert \frac{\partial H(x)}{\partial \nu} \Big\vert^{2}\ ds(x) \Big)^{\frac{1}{2}}= \Big(\int_{\partial D} \vert (\nabla_{x}H)(x).\nu \vert^{2}\ ds(x) \Big)^{\frac{1}{2}}=\Big(\delta^{2}.\delta^{-2} \int_{\partial B} \vert (\nabla_{\xi}\hat{H})(\xi).\nu \vert^{2}\ ds(\xi) \Big)^{\frac{1}{2}} \\
&=\Big\Vert \frac{\partial \hat{H}}{\partial \nu} \Big\Vert_{L^{2}(\partial B)}.
\end{aligned}
\label{estim-phi-2}
\end{equation}
From the estimate (see \eqref{H2propWk-f-eq-Elliptic-B}) $$\Vert \hat{H} \Vert_{H^{2}(B)} =\mathcal{O}(\delta^{2-\alpha}),$$
we have
\begin{equation}
\Vert \nabla \hat{H} \Vert_{H^{1}(B)} \leq \Vert \hat{H} \Vert_{H^{2}(B)} \leq C \delta^{2-\alpha}.
\notag
\end{equation}
Using trace theorem, it then follows that
\begin{equation}
\Vert \frac{\partial \hat{H}}{\partial \nu} \Vert_{H^{\frac{1}{2}}(\partial B)} \leq C \Vert \nabla \hat{H} \Vert_{H^{1}(B)} \leq C \delta^{2-\alpha}
\notag
\end{equation}
which implies, using \eqref{estim-phi-2},
\begin{equation}
\Vert \frac{\partial H}{\partial \nu} \Vert_{L^{2}(\partial D)}= \Vert \frac{\partial \hat{H}}{\partial \nu} \Vert_{L^{2}(\partial B)} \leq  \Vert \frac{\partial \hat{H}}{\partial \nu} \Vert_{H^{\frac{1}{2}}(\partial B)} \leq C \delta^{2-\alpha}.
\notag
\end{equation}
Using this in \eqref{estim-phi-1}, we have the required estimate.

%
%
\end{proof}
\begin{lemma}
$\Vert \psi \Vert_{L^{2}(\partial D)}=\mathcal{O}(1), \mbox{ if }\,\mu_1\sim \delta^{-\beta},\ \beta \geq 0,$ and $\kappa_{1}^{2} \sim \delta^{-\alpha},\ 0 \leq \alpha <1 $.
\end{lemma}
\begin{proof}
From the first identity in \eqref{sing-obs-1}, we can write
\begin{equation}
\psi=(S^{\kappa_{1}}_{D})^{-1} u^{I}+(S^{\kappa_{1}}_{D})^{-1} S^{\kappa_{0}}_{D} \phi.
\notag
\end{equation}
Therefore using \eqref{single}, we conclude
\begin{align*}
\Vert \psi \Vert_{L^{2}(\partial D)} &\leq\Vert (S^{\kappa_{1}}_{D})^{-1} \Vert_{\mathcal{L}(H^{1}(\partial D),L^{2}(\partial D))} \Vert u^{I} \Vert_{H^{1}(\partial D)} + \Vert (S^{\kappa_{1}}_{D})^{-1} S^{\kappa_{0}}_{D} \Vert_{\mathcal{L}(L^{2}(\partial D),L^{2}(\partial D))} \Vert \phi \Vert_{L^{2}(\partial D)} \\
&\leq C \delta^{-1} \delta + C \delta \Vert (S^{\kappa_{1}}_{D})^{-1} S^{\kappa_{0}}_{D} \Vert_{\mathcal{L}(L^{2}(\partial D),L^{2}(\partial D))} \leq C+ C \delta.
\end{align*}
\end{proof}
We can rewrite the first identity in \eqref{sing-obs-1} as
\begin{equation}
\begin{aligned}
S^{0}_{D}\psi(x)-S^{0}_{D}\phi(x)&=(S^{\kappa_{0}}_{D}\phi(x) -S^{0}_{D}\phi(x)) -(S^{\kappa_{1}}_{D}\psi(x)-S^{0}_{D} \psi(x)) +u^{I}(z)\\ 
&\ \ \ \ \ \ \ \ \ + \nabla u^{I}(z)(x-z)+\sum_{\vert j\vert=2}^{+\infty} \frac{1}{j !} \partial_{j} u^{I}(z) (x-z)^{j},
\end{aligned}
\label{sing-obs-50}  
\end{equation}
and the second identity in \eqref{sing-obs-1} as
\begin{equation}
\begin{aligned}
\frac{1}{\mu_{1}} [\frac{1}{2} Id+(K^{0}_{D})^{*}]\psi-\frac{1}{\mu_{0}} [-\frac{1}{2} Id+(K^{0}_{D})^{*}]\phi&=\frac{1}{\mu_{1}}[(K^{0}_{D})^{*}-(K^{\kappa_{1}}_{D})^{*}]\psi-\frac{1}{\mu_{0}}[(K^{0}_{D})^{*}-(K^{\kappa_{0}}_{D})^{*}]\phi \\
& \ \ \ \ \ \ +\frac{1}{\mu_{0}}[\nabla u^{I}(z).\frac{\partial}{\partial \nu}(x-z)+\sum_{\vert j \vert=2}^{+\infty} \frac{1}{j !} \partial_{j} u^{I}(z) \frac{\partial}{\partial \nu} (x-z)^{j}].
\end{aligned}
\label{sing-obs-51}
\end{equation}
\begin{lemma}
$\Vert S^{\kappa_{0}}_{D}\phi- S^{0}_{D}\phi \Vert_{H^{1}(\partial D)} = \mathcal{O}(\delta^3),\ \mbox{ if }\,\mu_1\sim \delta^{-\beta},\ \beta \geq 0,$ and $\kappa_{1}^{2} \sim \delta^{-\alpha},\ 0 \leq \alpha <1 $.
\end{lemma}
\begin{proof}
Using Lemma $3.5$, we can write
\begin{equation}
\Vert S^{\kappa_{0}}_{D}\phi- S^{0}_{D}\phi \Vert_{H^{1}(\partial D)}\leq \Vert S^{\kappa_{0}}_{D} -S^{0}_{D} \Vert_{\mathcal{L}(L^{2}(\partial D),H^{1}(\partial D))} \Vert \phi \Vert_{L^{2}(\partial D)} 
\leq C \delta^{2}. \delta= \mathcal{O}(\delta^{3}).
\notag
\end{equation}
\end{proof}
Let us also note that $u^{I}, \nabla u^{I} (x-z)$ are $\mathcal{O}(\delta)$ in $H^{1}(\partial D)$, $\sum_{\vert j \vert=2} \frac{1}{j !} \partial_{j} u^{I}(z) \frac{\partial}{\partial \nu} (x-z)^{j} $ 
is $\mathcal{O}(\delta^2)$ in $H^{1}(\partial D)$. Therefore ignoring the errors of $                                                                      \mathcal{O}(\delta^3) $ in \eqref{sing-obs-50}, we obtain
\begin{equation}
\begin{aligned}
S^{0}_{D}\psi(x)-S^{0}_{D}\phi(x)&= -(S^{\kappa_{1}}_{D}\psi(x)-S^{0}_{D} \psi(x)) +u^{I}(z)\\ 
&\ \ \ \ \ \ \ \ \ + \nabla u^{I}(z)(x-z)+\sum_{\vert j\vert=2} \frac{1}{j !} \partial_{j} u^{I}(z) (x-z)^{j}+ \mathcal{O}(\delta^{3}).
\end{aligned}
\label{sing-obs-52}  
\end{equation}
\begin{lemma}
$\Vert (K^{0}_{D})^{*} \phi-(K^{\kappa_{0}}_{D})^{*} \phi \Vert_{L^{2}(\partial D)}=  \mathcal{O}(\delta^3), \ \mbox{ if }\,\mu_1\sim \delta^{-\beta},\ \beta \geq 0,$ and $\kappa_{1}^{2} \sim \delta^{-\alpha},\ 0 \leq \alpha <1 $.
\end{lemma}
\begin{proof}
Using Lemma $3.3$, we can write
\begin{equation}
\Vert (K^{0}_{D})^{*} \phi-(K^{\kappa_{0}}_{D})^{*} \phi \Vert_{L^{2}(\partial D)}\leq \Vert (K^{0}_{D})^{*}-(K^{\kappa_{0}}_{D})^{*}\Vert_{\mathcal{L}(L^{2}(\partial D),L^{2}(\partial D))} \Vert \phi \Vert_{L^{2}(\partial D)} 
\leq C \delta^{2}. \delta= \mathcal{O}(\delta^{3}).
\notag
\end{equation}
\end{proof}
Again, noting that $\nabla u^{I}(z).\frac{\partial}{\partial \nu}(x-z), \sum_{\vert j \vert=2} \frac{1}{j !} \partial_{j} u^{I}(z) \frac{\partial}{\partial \nu} (x-z)^{j} $ are of order 
$\mathcal{O}(\delta) $ and $\mathcal{O}(\delta^2)$ in $L^{2}(\partial D)$ respectively, and ignoring the terms of order $\mathcal{O}(\delta^3)$ in \eqref{sing-obs-51}, we obtain
\begin{equation}
\begin{aligned}
\frac{1}{\mu_{1}} [\frac{1}{2} Id+(K^{0}_{D})^{*}]\psi-\frac{1}{\mu_{0}} [-\frac{1}{2} Id+(K^{0}_{D})^{*}]\phi=\frac{1}{\mu_{1}}[(K^{0}_{D})^{*}-(K^{\kappa_{1}}_{D})^{*}]\psi &+\frac{1}{\mu_{0}}[\nabla u^{I}(z).\frac{\partial}{\partial \nu}(x-z)\\
&+\sum_{\vert j \vert=2} \frac{1}{j !} \partial_{j} u^{I}(z) \frac{\partial}{\partial \nu} (x-z)^{j}]+\mathcal{O}(\delta^{3}).
\end{aligned}
\label{sing-obs-53}
\end{equation}
Keeping in mind the linearity of the identities, we split \eqref{sing-obs-52} and \eqref{sing-obs-53} into seven problems as following.\\
\\
\textbf{\textit{Problem P1}}:\\
\begin{equation}
\begin{aligned}
S^{0}_{D}\psi_{1}(x)-S^{0}_{D}\phi_{1}(x)&= u^{I}(z)+ \nabla u^{I}(z)(x-z),\\
\frac{1}{\mu_{1}} [\frac{1}{2} Id+(K^{0}_{D})^{*}]\psi_{1}-\frac{1}{\mu_{0}} [-\frac{1}{2} Id+(K^{0}_{D})^{*}]\phi_{1} &=\frac{1}{\mu_{0}}\nabla u^{I}(z).\frac{\partial}{\partial \nu}(x-z).
\end{aligned}
\label{p1}
\end{equation}
\textbf{\textit{Problem P2}}:\\
\begin{equation}
\begin{aligned}
S^{0}_{D}\psi_{2}(x)-S^{0}_{D}\phi_{2}(x)&=0,\\
\frac{1}{\mu_{1}} [\frac{1}{2} Id+(K^{0}_{D})^{*}]\psi_{2}-\frac{1}{\mu_{0}} [-\frac{1}{2} Id+(K^{0}_{D})^{*}]\phi_{2}&=\frac{1}{\mu_{0}}\sum_{\vert j \vert=2} \frac{1}{j !} \partial_{j} u^{I}(z) \frac{\partial}{\partial \nu} (x-z)^{j}.
\end{aligned}
\label{p2}
\end{equation}
\textbf{\textit{Problem P3}}:\\
\begin{equation}
\begin{aligned}
&S^{0}_{D}\psi_{3}(x)-S^{0}_{D}\phi_{3}(x)=\sum_{\vert j\vert=2} \frac{1}{j !} \partial_{j} u^{I}(z) (x-z)^{j}-\frac{1}{\vert \partial D \vert} \int_{\partial D} (\sum_{\vert j\vert=2} \frac{1}{j !} \partial_{j} u^{I}(z) (x-z)^{j})\ ds(x),\\
&\frac{1}{\mu_{1}} [\frac{1}{2} Id+(K^{0}_{D})^{*}]\psi_{3}-\frac{1}{\mu_{0}} [-\frac{1}{2} Id+(K^{0}_{D})^{*}]\phi_{3}=0.
\end{aligned}
\label{p3}
\end{equation}
\textbf{\textit{Problem P4}}:\\
\begin{equation}
\begin{aligned}
&S^{0}_{D}\psi_{4}(x)-S^{0}_{D}\phi_{4}(x)=\frac{1}{\vert \partial D \vert} \int_{\partial D} (\sum_{\vert j\vert=2} \frac{1}{j !} \partial_{j} u^{I}(z) (x-z)^{j})\ ds(x),\\
&\frac{1}{\mu_{1}} [\frac{1}{2} Id+(K^{0}_{D})^{*}]\psi_{4}-\frac{1}{\mu_{0}} [-\frac{1}{2} Id+(K^{0}_{D})^{*}]\phi_{4}=0.
\end{aligned}
\label{p4}
\end{equation}
\textbf{\textit{Problem P5}}:\\
\begin{equation}
\begin{aligned}
&S^{0}_{D}\psi_{5}(x)-S^{0}_{D}\phi_{5}(x)= -(S^{\kappa_{1}}_{D}\psi(x)-S^{0}_{D} \psi(x)) +\frac{1}{\vert \partial D \vert} \int_{\partial D} (S^{\kappa_{1}}_{D}\psi(x)-S^{0}_{D} \psi(x)) \ ds(x), \\
&\frac{1}{\mu_{1}} [\frac{1}{2} Id+(K^{0}_{D})^{*}]\psi_{5}-\frac{1}{\mu_{0}} [-\frac{1}{2} Id+(K^{0}_{D})^{*}]\phi_{5}=0.
\end{aligned}
\label{p5}
\end{equation}
\textbf{\textit{Problem P6}}:\\
\begin{equation}
\begin{aligned}
&S^{0}_{D}\psi_{6}(x)-S^{0}_{D}\phi_{6}(x)= -\frac{1}{\vert \partial D \vert} \int_{\partial D} (S^{\kappa_{1}}_{D}\psi(x)-S^{0}_{D} \psi(x)) \ ds(x), \\
&\frac{1}{\mu_{1}} [\frac{1}{2} Id+(K^{0}_{D})^{*}]\psi_{6}-\frac{1}{\mu_{0}} [-\frac{1}{2} Id+(K^{0}_{D})^{*}]\phi_{6}=\frac{1}{\mu_{1}}[(K^{0}_{D})^{*}-(K^{\kappa_{1}}_{D})^{*}]\psi.
\end{aligned}
\label{p6}
\end{equation}
\textbf{\textit{Problem P7}}:\\
\begin{equation}
\begin{aligned}
&S^{0}_{D}\psi_{7}(x)-S^{0}_{D}\phi_{7}(x)=\mathcal{O}(\delta^{3}), \\
&\frac{1}{\mu_{1}} [\frac{1}{2} Id+(K^{0}_{D})^{*}]\psi_{7}-\frac{1}{\mu_{0}} [-\frac{1}{2} Id+(K^{0}_{D})^{*}]\phi_{7}= \mathcal{O}( \delta^{3}).
\end{aligned}
\label{p7}
\end{equation}
Then upto an error of order $\mathcal{O}(\delta^3)$,
we can write
\begin{equation}
(\phi,\psi)=(\phi_{1},\psi_{1})+(\phi_{2},\psi_{2})+(\phi_{3},\psi_{3})+(\phi_{4},\psi_{4})+(\phi_{5},\psi_{5})+(\phi_{6},\psi_{6})+(\phi_{7},\psi_{7}).
\notag
\end{equation}
In the following, we shall treat each of these problems separately. We begin by noting the scaling property of the inverse of the single-layer potential.
\begin{lemma}\label{mean-zero}
Let $H^{1}_{\diamondsuit}(\partial D) $ denote the subspace of $H^{1}(\partial D) $ comprising the functions with zero mean. Then
\[\Vert (S^{0}_{D})^{-1} \Vert_{\mathcal{L}(H^{1}_{\diamondsuit}(\partial D),L^{2}(\partial D))} \leq C,\] where the constant $C$ is independent of $\delta$.
\end{lemma}
\begin{proof}
Let $f \in H^{1}_{\diamondsuit}(\partial D) $, that is, the mean $\bar{f}=0 $. Using the scaling properties (see \cite{Challa-Sini-1}) and surface Poincare inequality, we can write
\begin{equation}
\begin{aligned}
\Vert \hat{f} \Vert^{2}_{H^{1}(\partial B)} &=\int_{\partial B} \vert \hat{f}(\xi) \vert^{2} \ ds(\xi) + \int_{\partial B} \vert \partial_{T} \hat{f}(\xi) \vert^{2}\ ds(\xi) =\int_{\partial B} \vert \hat{f}(\xi)-\bar{\hat{f}}(\xi) \vert^{2} \ ds(\xi) + \int_{\partial B} \vert \partial_{T} \hat{f}(\xi) \vert^{2} \ ds(\xi) \\
&\leq C \int_{\partial B} \vert \partial_{T} \hat{f}(\xi) \vert^{2}\ ds(\xi) \leq C \int_{\partial D} \vert \partial_{T} f(x) \vert^{2}\ ds(x)= C \Vert f \Vert^{2}_{H^{1}(\partial D)},
\end{aligned}
\notag
\end{equation}
which implies \[\Vert \hat{f} \Vert_{H^{1}_{\diamondsuit}(\partial B)} \leq C \Vert f \Vert_{H^{1}_{\diamondsuit}(\partial D)}. \]
Now let $\phi \in L^{2}(\partial D)$ be such that $S^{0}_{D} \phi=f $. Then
\[\delta \ S^{0}_{B} \hat{\phi}=\hat{f} \ \text{on}\ \partial B, \] 
and we can write
\[\delta \Vert \hat{\phi} \Vert_{L^{2}(\partial B)} = \Vert (S^{0}_{B})^{-1} \hat{f} \Vert_{L^{2}(\partial B)} \leq C \Vert \hat{f} \Vert_{H^{1}_{\diamondsuit}(\partial B)}, \]
where $C$ is independent of $\delta $. Therefore
\[\Vert \phi \Vert_{L^{2}(\partial D)} = \delta \Vert \hat{\phi} \Vert_{L^{2}(\partial B)} \leq C \Vert \hat{f} \Vert_{H^{1}_{\diamondsuit}(\partial B)} \leq C \Vert f \Vert_{H^{1}_{\diamondsuit}(\partial D)}, \]
and hence $\Vert (S^{0}_{D})^{-1} \Vert_{\mathcal{L}(H^{1}_{\diamondsuit}(\partial D),L^{2}(\partial D))} \leq C $.
\end{proof}
\subsubsection{Treatment of Problem P1}
We rewrite the first identity in P1 as
\begin{equation}
[\underbrace{S^{0}_{D} \psi_{1}(x)-S^{0}_{D} \phi_{1}(x)}_{H1}]-[\underbrace{u^{I}(z)+ \nabla u^{I}(z)(x-z)}_{H2}]=0 \ \text{on} \ \partial D.
\notag
\end{equation}
Let us also note that both $H1$ and $H2$ are harmonic in $x$, and hence satisfies
\begin{equation}
\Delta_{x}(H1-H2)=0 \ \text{in} \ D.
\notag
\end{equation}
Therefore by maximum principle, it follows that $H1-H2=0 $ in $D$, that is,
\begin{equation}
S^{0}_{D} \psi_{1}(x)-S^{0}_{D} \phi_{1}(x)= u^{I}(z)+ \nabla u^{I}(z)(x-z) \ \text{in}\ D.
\notag
\end{equation}
Taking the normal derivative from the interior of the domain $D$, we then have
\begin{equation}
\partial_{\nu} S^{0}_{D} \psi_{1} \Big\vert_{-} - \partial_{\nu} S^{0}_{D} \phi_{1} \Big\vert_{-} = \sum_{i=1}^{3} \partial_{i} u^{I}(z)\ \nu_{i}.
\notag
\end{equation}
This can also be written as
\begin{equation}
[\frac{1}{2} Id+ (K^{0}_{D})^{*}] \psi_{1}(x)-[\frac{1}{2} Id+ (K^{0}_{D})^{*}] \phi_{1}(x) =\sum_{i=1}^{3} \partial_{i} u^{I}(z)\ \nu_{i}. 
\label{p1-1}
\end{equation}
Combining this with the second identity of P1, we then obtain
\begin{equation}
\begin{aligned}
\Big[-\frac{1}{2} Id - (K^{0}_{D})^{*}+ \frac{\mu_{1}}{\mu_{0}} \Big[-\frac{1}{2} Id + (K^{0}_{D})^{*} \Big] \Big] \phi_{1} &= (1-\frac{\mu_{1}}{\mu_{0}}) \sum_{i=1}^{3} \partial_{i} u^{I}(z)\ \nu_{i} \\
\Rightarrow \Big[\lambda Id+ (K^{0}_{D})^{*} \Big] \phi_{1} = - \sum^{3}_{i=1} \partial_{i} u^{I}(z)\ \nu_{i},
\end{aligned}
\notag
\end{equation}
where $\lambda=\frac{1}{2} \frac{\mu_{0}+\mu_{1}}{\mu_{0}-\mu_{1}}$. Now since $\mu_{0} \neq \mu_{1}$, the operator $[\lambda Id+ (K^{0}_{D})^{*}]$ is invertible and hence we can write
\begin{equation}
\begin{aligned}
\phi_{1}&= - \Big[\lambda Id+ (K^{0}_{D})^{*} \Big]^{-1} (\sum_{i=1}^{3} \partial_{i} u^{I}(z) \nu_{i} ) \\
\Rightarrow \phi_{1}&= -\sum_{i=1}^{3} \partial_{i} u^{I}(z) \Big[\lambda Id+ (K^{0}_{D})^{*} \Big]^{-1} \nu_{i} . 
\end{aligned}
\label{p1-2}
\end{equation}
Next we note that since $\phi_{1}$ satisfies the first identity in \eqref{p1}, we can also write
\begin{equation}
\begin{aligned}
\phi_{1}&= \partial_{\nu} S^{0}_{D} \phi_{1} \Big\vert_{-} - \partial_{\nu} S^{0}_{D} \phi_{1} \Big\vert_{+} \\
&= \partial_{\nu} S^{0}_{D} \phi_{1} \Big\vert_{-} -\Big[\frac{\mu_{0}}{\mu_{1}} \partial_{\nu} S^{0}_{D} \psi_{1} \Big\vert_{-}- \sum_{i=1}^{3} \partial_{i} u^{I}(z) \nu_{i} \Big]  \\
&= \partial_{\nu} S^{0}_{D} \phi_{1} \Big\vert_{-} -\frac{\mu_{0}}{\mu_{1}}  \partial_{\nu} S^{0}_{D} \psi_{1} \Big\vert_{-} + \sum_{i=1}^{3} \partial_{i} u^{I}(z) \nu_{i}.
\end{aligned}
\notag
\end{equation}
Now using the fact $\nu_{i}(x)=\partial_{\nu} x_{i}$ and the Green's formula, we have
\begin{equation}
\begin{aligned}
\int_{\partial D} \phi_{1}&= \int_{D} \Delta S^{0}_{D} \phi_{1}(x) dx - \frac{\mu_{0}}{\mu_{1}} \int_{D} \Delta S^{0}_{D} \psi_{1}(x) dx+ \sum_{i=1}^{3} \partial_{i} u^{I}(z) \int_{D} \Delta x_{i} dx 
=0.
\end{aligned}
\label{p1-3}
\end{equation}
Let us now estimate $\int_{\partial D} (x-z)^{j} \phi_{1}(x) ds(x)$, for $\vert j \vert=1 $ where $(x-z)^{j}=x_{j}-z_{j} $. \\
Recalling \eqref{p1-2}, we have
\begin{equation}
\begin{aligned}
\int_{\partial D} (x-z)^{j} \phi_{1}(x) ds(x) &=- \sum_{i=1}^{3} \partial_{i} u^{I}(z) \int_{\partial D} (x-z)^{j} \Big[\lambda Id+(K^{0}_{D})^{*} \Big]^{-1}(\nu_i)(x) \ ds(x) \\
&=-\sum_{i=1}^{3} \partial_{i} u^{I}(z) \delta^{3} \int_{\partial B}y^j[\lambda Id+(K_{B}^0)^*]^{-1}(\nu_x\cdot \nabla x^i)(y)\  ds(y) \\ 
&(= \mathcal{O}(\delta^{3})).
\end{aligned}
\label{p1-4}
\end{equation}
%
\subsubsection{Treatment of Problem P2}
As in the case of Problem $P1$, we use the harmonicity of the right hand side of the first identity in \eqref{p2} and proceed as following.\\
Let $H3(x)=S^{0}_{D} \psi_{2}(x)-S^{0}_{D} \phi_{2}(x)$. Then $H3$ satisfies
\begin{equation}
\begin{aligned}
\Delta_{x} H3(x) &= 0 \ \text{in} \ D, \\
H3(x)&=0 \ \text{on} \ \partial D. 
\end{aligned}
\notag
\end{equation}
Therefore by the maximum principle, we infer that $H3(x)=0 \ \text{in} \ D $, and hence by taking the normal derivative from the interior we obtain
\begin{equation}
\begin{aligned}
\partial_{\nu} S^{0}_{D} \psi_{2}\Big\vert_{-} (x)-\partial_{\nu} S^{0}_{D} \phi_{2}\Big\vert_{-} (x)&=0 \\
\Rightarrow \Big[\frac{1}{2} Id+ (K^{0}_{D})^{*} \Big] \psi_{2}(x) &= \Big[\frac{1}{2} Id+ (K^{0}_{D})^{*} \Big] \phi_{2}(x). 
\end{aligned}
\label{p2-1}
\end{equation}
Using this in the second identity of \eqref{p2}, we have
\begin{equation}
\begin{aligned}
&\frac{1}{\mu_{1}} \Big[\frac{1}{2} Id+ (K^{0}_{D})^{*} \Big]\phi_{2}(x) - \frac{1}{\mu_{0}} \Big[-\frac{1}{2} Id+ (K^{0}_{D})^{*} \Big]\phi_{2}(x) = \sum_{\vert j \vert=2} \frac{1}{j!} \partial_{j} u^{I}(z) \frac{\partial}{\partial \nu} (x-z)^{j} \\
&\Rightarrow \frac{\mu_{0}-\mu_{1}}{\mu_{0} \mu_{1}} \Big[\lambda \ Id+(K^{0}_{D})^{*} \Big] \phi_{2}(x) = \sum_{\vert j \vert=2} \frac{1}{j!} \partial_{j} u^{I}(z) \frac{\partial}{\partial \nu} (x-z)^{j} \\
&\Rightarrow \phi_{2}(x) = \frac{\mu_{0} \mu_{1}}{\mu_{0}-\mu_{1}} \sum_{\vert j \vert=2 } \frac{1}{j!} \partial_{j} u^{I}(z) \Big[\lambda \ Id+ (K^{0}_{D})^{*} \Big]^{-1} \frac{\partial}{\partial \nu} (x-z)^{j}. 
\end{aligned}
\label{p2-2}
\end{equation}
Next we note that we can also write
\begin{equation}
\begin{aligned}
\phi_{2}(x) &= \partial_{\nu} S^{0}_{D} \phi_{2} \Big\vert_{-} (x)- \partial_{\nu} S^{0}_{D} \phi_{2} \Big\vert_{+}(x) \\
&=\partial_{\nu} S^{0}_{D} \phi_{2} \Big\vert_{-}(x)-\frac{\mu_{0}}{\mu_{1}} \partial_{\nu} S^{0}_{D} \psi_{2}\Big\vert_{-} (x)  + \sum_{\vert j \vert=2} \frac{1}{j !} \partial_{j} u^{I}(z) \frac{\partial}{\partial \nu} (x-z)^{j}.
\end{aligned}
\label{p2-3}
\end{equation}
Then using Green's formula, we have
\begin{equation}
\begin{aligned}
\int_{\partial D} \phi_{2}(x) \ ds(x) &= \int_{D} \Delta S^{0}_{D} \phi_{2}(x) \ dx - \frac{\mu_{0}}{\mu_{1}} \int_{D} \Delta S^{0}_{D} \psi_{2}(x) \ dx + \sum_{\vert j \vert=2} \frac{1}{j !} \partial_{j} u^{I}(z) \int_{\partial D} \frac{\partial}{\partial \nu} (x-z)^{j} \ ds(x) \\
&= \sum_{\vert j \vert=2} \frac{1}{j !} \partial_{j} u^{I}(z) \int_{\partial D} \frac{\partial}{\partial \nu} (x-z)^{j} \ ds(x).
\end{aligned}
\notag
\end{equation}
Now 
\begin{equation}
\begin{aligned}
\sum_{\vert j \vert=2} \frac{1}{j !} \partial_{j} u^{I}(z) \int_{\partial D} \frac{\partial}{\partial \nu} (x-z)^{j} \ ds(x) &= \sum_{\vert j \vert=2} \frac{1}{j !} \partial_{j} u^{I}(z) \int_{ D} \Delta (x-z)^{j} \ dx \\
&=\sum_{j=1}^{3} \frac{1}{2 !} \partial^{2}_{j} u^{I}(z). 2 \vert D \vert = \Delta u^{I}(z) \delta^{3} \vert B \vert \\
&=- \kappa_{0}^{2} u^{I}(z) \vert B \vert \delta^{3},
\end{aligned}
\notag
\end{equation} 
and therefore
\begin{equation}
\int_{\partial D} \phi_{2}(x) \ ds(x) = - \kappa_{0}^{2} u^{I}(z) \vert B \vert \delta^{3}.
\label{p2-4}
\end{equation}
Next we estimate $\int_{\partial D} (x-z)^{i} \phi_{2}(x) \ ds(x)$ for $\vert j \vert=1 $. \\
Using Cauchy-Schwarz inequality, we have
\begin{equation}
\begin{aligned}
\int_{\partial D} (x-z)^{i} \phi_{2}(x) \ ds(x) &\leq \Vert (x-z)^{i} \Vert_{L^{2}(\partial D)} \Vert \phi_{2} \Vert_{L^{2}(\partial D)} \\
&\leq \Vert (x_{i}-z_{i}) \Vert_{L^{2}(\partial D)} \Vert \phi_{2} \Vert_{L^{2}(\partial D)}.
\end{aligned}
\notag
\end{equation}
Now $\Vert (x_{i}-z_{i}) \Vert_{L^{2}(\partial D)}= \mathcal{O}(\delta^{2})$ and from \eqref{p2-2}, using Minkowski's inequality, we have
\begin{equation}
\begin{aligned}
\Vert \phi_{2} \Vert_{L^{2}(\partial D)} &\leq \vert \frac{\mu_{0}\mu_{1}}{\mu_{0}-\mu_{1}} \vert \sum_{\vert j \vert=2} \frac{1}{j !} \vert \partial_{j} u^{I}(z) \vert \Vert \Big[\lambda Id+(K^{0}_{D})^{*} \Big]^{-1} \frac{\partial}{\partial \nu} (x-z)^{j}\Vert_{L^{2}(\partial D)}\\
&\leq \vert \frac{\mu_{0} \mu_{1}}{\mu_{0}-\mu_{1}} \vert \Vert \Big[\lambda Id+ (K^{0}_{D})^{*} \Big]^{-1} \Vert_{L^{2}(\partial D)} \sum_{\vert j \vert=2} \frac{1}{j !} \vert \partial_{j} u^{I}(z) \vert \Vert \frac{\partial}{\partial \nu} (x-z)^{j} \Vert_{L^{2}(\partial D)}. 
\end{aligned}
\notag
\end{equation}
Thus we need to estimate $\Vert \frac{\partial}{\partial \nu} (x-z)^{j} \Vert_{L^{2}(\partial D)} $. But
\begin{equation}
\begin{aligned}
\vert \frac{\partial}{\partial \nu} (x-z)^{j} \vert^{2} &\leq C \vert x-z \vert^{2} \\
\Rightarrow \int_{\partial D} \vert \frac{\partial}{\partial \nu} (x-z)^{j} \vert^{2} \ ds(x) &\leq C \int_{\partial D} \vert x-z \vert^{2} \ ds(x) \leq C \delta^{2} \vert \partial D \vert \leq C \delta^{4} \\
\Rightarrow \Vert \frac{\partial}{\partial \nu} (x-z)^{j} \Vert_{L^{2}(\partial D)} &\leq (C \delta^{4})^{\frac{1}{2}} =\mathcal{O}(\delta^{2}),
\end{aligned}
\notag
\end{equation}
and therefore we obtain that 
\begin{equation}
\int_{\partial D} (x-z)^{i} \phi_{2}(x) \ ds(x)= \mathcal{O}(\delta^{4}).
\label{p2-5}
\end{equation}
\subsubsection{Treatment of Problem P3}
In this case, using the second identity of \eqref{p3} we write
\begin{equation}
\begin{aligned}
\phi_{3} &=\partial_{\nu} S^{0}_{D} \phi_{3} \Big\vert_{-}-\partial_{\nu} S^{0}_{D} \phi_{3} \Big\vert_{+} \\
&=\partial_{\nu} S^{0}_{D} \phi_{3} \Big\vert_{-}- \frac{\mu_{0}}{\mu_{1}} \Big[\frac{1}{2} Id+(K^{0}_{D})^{*} \Big] \psi_{3} \\
&=\partial_{\nu} S^{0}_{D} \phi_{3} \Big\vert_{-} -\frac{\mu_{0}}{\mu_{1}} \partial_{\nu} S^{0}_{D} \psi_{3} \Big\vert_{-} , 
\end{aligned}
\notag
\end{equation}
and hence using Green's identity, we have
\begin{equation}
\int_{\partial D} \phi_{3}(x) \ ds(x) = \int_{D} \Delta S^{0}_{D} \phi_{3}(x) \ dx- \frac{\mu_{0}}{\mu_{1}} \int_{D} \Delta S^{0}_{D} \psi_{3}(x) \ dx =0.
\label{p3-1}
\end{equation}
Let us now estimate $\int_{\partial D} (x-z)^{i} \phi_{3}(x) \ ds(x)$ for $\vert i \vert=1$. Using Cauchy-Schwarz inequality, we have
\begin{equation}
\Big\vert \int_{\partial D} (x-z)^{i} \phi_{3}(x) \ ds(x) \Big\vert \leq \Vert (x-z)^{i} \Vert_{L^{2}(\partial D)} \Vert \phi_{3} \Vert_{L^{2}(\partial D)}
\notag
\end{equation}
and also we have already seen that $\Vert (x-z)^{i} \Vert_{L^{2}(\partial D)}= \mathcal{O}(\delta^{2}) $. 
%
%
Let us denote $f(x) := \sum_{\vert j \vert=2} \frac{1}{j !} \partial_{j} u^{I}(z) (x-z)^{j} $ and $c:=\frac{1}{\vert \partial D \vert} \int_{\partial D} f(x) \ ds(x)$, that is, $c$ is the average of $f$ over the boundary $\partial D$ of $D$.Then  
\begin{equation}
\frac{1}{\vert \partial D \vert} \int_{\partial D} f = \frac{\delta^{2}}{\vert \partial D \vert} (\int_{\partial B} \hat{f}) = C_{1} (\int_{\partial B} \hat{f}), 
\notag
\end{equation}
where $C_{1}$ is independent of $\delta$.\\
Now 
\begin{equation}
\begin{aligned}
\hat{f}(\xi) &= \sum_{\vert j \vert=2} \frac{1}{j !} \partial_{j} u^{I}(z) (\delta \xi)^{j}= \delta^{2} \sum_{\vert j \vert=2} \frac{1}{j !} \partial_{j} u^{I}(z) \xi^{j} \\
\Rightarrow \int_{\partial B} \hat{f} &= \delta^{2} \Big[ \int_{\partial B} \sum_{\vert j \vert=2} \frac{1}{j !} \partial_{j} u^{I}(z) \xi^{j} \Big], 
\end{aligned}
\notag
\end{equation}
whereby we obtain that $c= C_{2} \delta^{2}$, where the constant $C_{2}$ is now independent of $\delta$.\\
Now using the invertibility of  $(S^{0}_{D})^{-1}$, we can rewrite the first identity in \eqref{p3} as
\begin{equation}
\psi_{3}= \phi_{3}+ (S^{0}_{D})^{-1} (f-c).
\notag
\end{equation}
Using this in the second identity of \eqref{p3}, we obtain
\begin{equation}
\begin{aligned}
&\Big[-\frac{1}{2} Id+(K^{0}_{D})^{*} \Big] \phi_{3}= \frac{\mu_{0}}{\mu_{1}} \Big[\frac{1}{2} Id+(K^{0}_{D})^{*} \Big] (S^{0}_{D})^{-1} (f-c) + \frac{\mu_{0}}{\mu_{1}} \Big[\frac{1}{2} Id+ (K^{0}_{D})^{*} \Big] \phi_{3} \\
&\Rightarrow \Big[\Big[-\frac{1}{2} Id+(K^{0}_{D})^{*} \Big]-\frac{\mu_{0}}{\mu_{1}} \Big[\frac{1}{2} Id+(K^{0}_{D})^{*} \Big] \Big] \phi_{3} = \frac{\mu_{0}}{\mu_{1}} \Big[\frac{1}{2} Id+(K^{0}_{D})^{*} \Big] (S^{0}_{D})^{-1} (f-c) \\
&\Rightarrow \phi_{3} = \frac{\mu_{0}}{\mu_{1}-\mu_{0}} \Big[-\frac{\mu_{0}+\mu_{1}}{2(\mu_{1}-\mu_{0})} Id+ (K^{0}_{D})^{*} \Big]^{-1} \Big[\frac{1}{2} Id+ (K^{0}_{D})^{*} \Big] (S^{0}_{D})^{-1} (f-c),
\end{aligned} 
\notag
\end{equation}
whereby we obtain
\begin{equation}
\begin{aligned}
\Vert \phi_{3} \Vert_{L^{2}(\partial D)} \leq C \Big\vert\frac{\mu_0}{\mu_1-\mu_0}\Big\vert \Big\Vert \Big[-\frac{\mu_{0}+\mu_{1}}{2(\mu_{1}-\mu_{0})} Id+ (K^{0}_{D})^{*} \Big]^{-1} \Big\Vert_{\mathcal{L}(L^{2}(\partial D),L^{2}(\partial D))} &\Big\Vert \Big[ \frac{1}{2} Id+(K^{0}_{D})^{*} \Big] \Big\Vert_{\mathcal{L}(L^{2}(\partial D),L^{2}(\partial D))} \\
\Big\Vert (S^{0}_{D})^{-1} \Big\Vert_{\mathcal{L}(H^{1}(\partial D),L^{2}(\partial D))} &\Vert f-c \Vert_{H^{1}(\partial D)} \\
\leq C \frac{\mu_0}{\mu_1} \Vert f-c \Vert_{H^{1}(\partial D)} \leq C \frac{\mu_0}{\mu_1}(\Vert f \Vert_{H^{1}(\partial D)}+ \Vert c \Vert_{H^{1}(\partial D)}) &\leq C \frac{\mu_0}{\mu_1} (\delta^{2}+ \delta^{3}),
\end{aligned}
\notag
\end{equation}
since $\Vert c \Vert_{H^{1}(\partial D)}=\Vert c \Vert_{L^{2}(\partial D)}=\mathcal{O}(\delta^{3}),$ and $(S^{0}_{D})^{-1} $ is uniformly bounded with respect to $\delta $ when acting over the mean-zero function $f-c $ (see Lemma \ref{mean-zero}). 
 Therefore $\Vert \phi_{3} \Vert_{L^{2}(\partial D)}=\mathcal{O}(\delta^{2})$ and hence 
\begin{equation}
\Big\vert \int_{\partial D} (x-z)^{i} \phi_{3}(x)\ ds(x) \Big\vert = \mathcal{O}(\frac{\mu_0}{\mu_1}\delta^{4}).
\label{p3-2}
\end{equation}
\subsubsection{Treatment of Problem P4}
In this case, we note that the right hand side of \eqref{p4} is a constant, and hence independent of $x$. Therefore we can again use the harmonicity of both sides of \eqref{p4} together with maximum principle to obtain the identity
\begin{equation}
S^{0}_{D}\psi_{4}(x)-S^{0}_{D} \phi_{4}(x)= c \ \text{in}\ D, 
\notag
\end{equation}
where $c=\frac{1}{\vert \partial D \vert} \int_{\partial D} (\sum_{\vert j\vert=2} \frac{1}{j !} \partial_{j} u^{I}(z) (x-z)^{j})\ ds(x) $.\\
Taking the normal derivative from the interior of the domain, we then obtain
\begin{equation}
\partial_{\nu} S^{0}_{D} \psi_{4} \Big\vert_{-} -\partial_{\nu} S^{0}_{D} \phi_{4} \Big\vert_{-}=0,
\notag
\end{equation}
which implies
\begin{equation}
\Big[\frac{1}{2} Id+ (K^{0}_{D})^{*} \Big] \psi_{4} = \Big[\frac{1}{2} Id+ (K^{0}_{D})^{*} \Big] \phi_{4}.
\notag
\end{equation}
Using the second identity in \eqref{p4}, we then obtain
\begin{equation}
\begin{aligned}
\Big[\frac{\mu_{1}}{\mu_{0}} \Big[-\frac{1}{2} Id+(K^{0}_{D})^{*} \Big]-\Big[\frac{1}{2} Id+ (K^{0}_{D})^{*} \Big] \Big]\phi_{4}&=0 \\
\Rightarrow \Big[ (-\frac{\mu_{1}}{2 \mu_{0}}-\frac{1}{2}) Id+ (\frac{\mu_{1}}{\mu_{0}}-1) (K^{0}_{D})^{*} \Big] \phi_{4} &=0 \\
\Rightarrow \Big[\lambda Id+ (K^{0}_{D})^{*} \Big] \phi_{4}&=0 \\
\Rightarrow \phi_{4} =0,
\end{aligned}
\notag
\end{equation}
since $\lambda Id+ (K^{0}_{D})^{*} $ is invertible where $\lambda=\frac{1}{2}. \frac{\mu_{0}+\mu_{1}}{\mu_{0}-\mu_{1}} $.\\
This then implies
\begin{equation}
\int_{\partial D} \phi_{4}(x) \ ds(x)=0, \ \int_{\partial D} (x-z)^{i} \phi_{4}(x)\ ds(x)=0.
\label{p4-1}
\end{equation} 
\subsubsection{Treatment of Problem P5}
In this case, using the second identity in \eqref{p5}, we can write
\begin{equation}
\begin{aligned}
\phi_{5} &= \partial_{\nu} S^{0}_{D} \phi_{5} \Big\vert_{-} - \partial_{\nu} S^{0}_{D} \phi_{5} \Big\vert_{+} \\
&=\partial_{\nu} S^{0}_{D} \phi_{5} \Big\vert_{-} -\frac{\mu_{0}}{\mu_{1}} \Big[\frac{1}{2} Id+(K^{0}_{D})^{*} \Big] \psi_{5} \\
&=\partial_{\nu} S^{0}_{D} \phi_{5} \Big\vert_{-} -\frac{\mu_{0}}{\mu_{1}} \partial_{\nu} S^{0}_{D} \psi_{5} \Big\vert_{-}.  
\end{aligned}
\notag
\end{equation}
Using Green's formula, we can then write
\begin{equation}
\begin{aligned}
\int_{\partial D} \phi_{5}(x)\ ds(x) &= \int_{D} \Delta S^{0}_{D} \phi_{5} (x) \ dx- \frac{\mu_{0}}{\mu_{1}} \int_{D} \Delta S^{0}_{D} \psi_{5}(x) \ dx \\
&=0. 
\end{aligned}
\label{p5-1}
\end{equation}
Let us write $\tilde{f}(x):=-(S^{\kappa_{1}}_{D} \psi(x)-S^{0}_{D} \psi(x))$ and $\tilde{c}:=-\frac{1}{\vert \partial D \vert} \int_{\partial D} (S^{\kappa_{1}}_{D} \psi(x)-S^{0}_{D} \psi(x))\ ds(x) $.\\
Then proceeding as in the case of Problem $P3$, we can write
\begin{equation}
\phi_{5}=\frac{\mu_{0}}{\mu_{1}-\mu_{0}} \Big[\lambda Id+(K^{0}_{D})^{*} \Big]^{-1} \Big[\frac{1}{2} Id+(K^{0}_{D})^{*} \Big] (S^{0}_{D})^{-1}[\tilde{f}-\tilde{c}].
\notag
\end{equation}
Then
\begin{equation}
\Vert \phi_{5} \Vert_{L^{2}(\partial D)} \leq C \frac{\mu_0}{\mu_1}(\Vert S^{\kappa_{1}}_{D} \psi-S^{0}_{D} \psi \Vert_{H^{1}(\partial D)} + \Vert \tilde{c} \Vert_{H^{1}(\partial D)}) \leq C \frac{\mu_0}{\mu_1}(\delta^{2-\alpha}+ \Vert \tilde{c} \Vert_{H^{1}(\partial D)}).
\notag
\end{equation}
Now 
\begin{equation}
\tilde{c}=\frac{1}{\vert \partial D \vert} \int_{\partial D} (S^{\kappa_{1}}_{D} \psi- S^{0}_{D} \psi)= C \int_{\partial B} ({S^{\kappa_{1}}_{D} \psi-S^{0}_{D} \psi})^{\wedge} = C \delta \int_{\partial B} (S^{\kappa_{1},\delta}_{B} \hat{\psi}- S^{0,\delta}_{B} \hat{\psi}),
\label{p5-2}
\end{equation}
where $S^{\kappa_{1},\delta}_{B} \hat{\psi}(\xi)=\int_{\partial B} \frac{e^{i \kappa_{1} \delta \vert \xi-\eta \vert}}{4 \pi \vert \xi-\eta \vert} \hat{\psi}(\eta)\ ds(\eta) $.\\
Using Cauchy-Schwarz inequality, we can estimate
\begin{equation}
\vert \int_{\partial B} (S^{\kappa_{1},\delta}_{B} \hat{\psi}(\xi)-S^{0,\delta}_{B} \hat{\psi}(\xi) )\ ds(\xi) \vert \leq \Vert S^{\kappa_{1},\delta}_{B} \hat{\psi}(\xi)-S^{0,\delta}_{B} \hat{\psi}(\xi) \Vert_{L^{2}(\partial B)} \Vert 1 \Vert_{L^{2}(\partial B)} 
\label{p5-3}
\end{equation}
Next we note that
\begin{equation}
\begin{aligned}
\vert S^{\kappa_{1},\delta}_{B} \hat{\psi}(\xi)- S^{0,\delta}_{B}\hat{\psi}(\xi) \vert &\leq \int_{\partial B} \Big\vert \frac{e^{i \kappa_{1} \delta \vert \xi-\eta \vert}-1}{4 \pi \vert \xi-\eta \vert} \Big\vert \vert \hat{\psi}(\eta) \vert \ ds(\eta) \\
&\leq \Big\Vert \frac{e^{i \kappa_{1} \delta \vert \xi-. \vert}-1}{4 \pi \vert \xi-. \vert} \Big\Vert_{L^{2}(\partial B)} \Vert \hat{\psi} \Vert_{L^{2}(\partial B)} \\
&\leq \delta^{-1} \Vert \psi \Vert_{L^{2}(\partial D)} \Big\Vert \frac{e^{i \kappa_{1} \delta \vert \xi-. \vert}-1}{4 \pi \vert \xi-. \vert} \Big\Vert_{L^{2}(\partial B)} \\
&\leq C \delta^{-1} \Big\Vert \frac{e^{i \kappa_{1} \delta \vert \xi-. \vert}-1}{4 \pi \vert \xi-. \vert} \Big\Vert_{L^{2}(\partial B)}
\end{aligned}
\notag
\end{equation}
which implies
\begin{equation}
\begin{aligned}
\Rightarrow \Vert S^{\kappa_{1},\delta}_{B} \hat{\psi}-S^{0,\delta}_{B} \hat{\psi} \Vert_{L^{2}(\partial B)} &= (\int_{\partial B} \vert S^{\kappa_{1},\delta}_{B} \hat{\psi}(\xi)-S^{0,\delta}_{B}\hat{\psi}(\xi) \vert^{2} \ ds(\xi))^{\frac{1}{2}} \\
&\leq C \delta^{-1} (\int_{\partial B} \Big\Vert \frac{e^{i \kappa_{1} \delta \vert \xi-. \vert}-1}{4 \pi \vert \xi-. \vert} \Big\Vert^{2}_{L^{2}(\partial B)}\ ds(\xi))^{\frac{1}{2}}.
\end{aligned}
\notag
\end{equation}
But
\begin{equation}
\frac{e^{i \kappa_{1} \delta \vert \xi-\eta \vert}-1}{4 \pi \vert \xi-\eta \vert} = \mathcal{O}(\kappa_{1} \delta)=\mathcal{O}(\delta^{1-\frac{\alpha}{2}}),
\notag
\end{equation}
and therefore
\begin{equation}
\Vert S^{\kappa_{1},\delta}_{B}\hat{\psi}-S^{0,\delta}_{B} \hat{\psi} \Vert_{L^{2}(\partial B)} \leq C \delta^{-1} \delta^{1-\frac{\alpha}{2}} =C \delta^{-\frac{\alpha}{2}}.
\notag
\end{equation}
Using this in \eqref{p5-3}, we get
\begin{equation}
\vert \int_{\partial B} (S^{\kappa_{1},\delta}_{B} \hat{\psi}(\xi)-S^{0,\delta}_{B}\hat{\psi}(\xi))\ ds(\xi) \vert \leq C \delta^{-\frac{\alpha}{2}},
\label{p5-4}
\end{equation}
where $C$ is independent of $\delta$.
We then have
\begin{equation}
\Vert \tilde{c} \Vert_{H^{1}(\partial D)}=\Vert \tilde{c} \Vert_{L^{2}(\partial D)}=C \delta^{1-\frac{\alpha}{2}} \vert \partial D \vert^{\frac{1}{2}}= C \delta^{2-\frac{\alpha}{2}},
\notag
\end{equation}
and hence $\Vert \phi_{5} \Vert_{L^{2}(\partial D)} \leq C \frac{\mu_0}{\mu_1} \delta^{2-\alpha}$.\\
Therefore, using Cauchy-Schwarz inequality, it follows that 
\begin{equation}
\int_{\partial D} (x-z)^{j} \phi_{5}= \mathcal{O}(\frac{\mu_0}{\mu_1} \delta^{4-\alpha}) .
\label{p5-5}
\end{equation}
\subsubsection{Treatment of Problem P6}
In this case, using the second identity in \eqref{p6}, we can write
\begin{equation}
\begin{aligned}
\phi_{6} &= \partial_{\nu} S^{0}_{D} \phi_{6} \Big\vert_{-} - \partial_{\nu} S^{0}_{D} \phi_{6} \Big\vert_{+} \\
&=\partial_{\nu} S^{0}_{D} \phi_{6} \Big\vert_{-} -\frac{\mu_{0}}{\mu_{1}} \Big[\frac{1}{2} Id+(K^{0}_{D})^{*} \Big] \psi_{6} +\frac{\mu_{0}}{\mu_{1}} \Big[(K^{0}_{D})^{*}-(K^{\kappa_{1}}_{D})^{*} \Big] \psi  \\
&=\partial_{\nu} S^{0}_{D} \phi_{6} \Big\vert_{-} -\frac{\mu_{0}}{\mu_{1}} \partial_{\nu} S^{0}_{D} \psi_{6} \Big\vert_{-}+ \frac{\mu_{0}}{\mu_{1}} \partial_{\nu} S^{0}_{D} \psi \Big\vert_{-} -\frac{\mu_{0}}{\mu_{1}} \Big[\frac{1}{2} Id+ (K^{\kappa_{1}}_{D})^{*} \Big] \psi.  
\end{aligned}
\notag
\end{equation}
Using Green's formula, we can then write
\begin{equation}
\begin{aligned}
\int_{\partial D} \phi_{6}(x)\ ds(x) &= \int_{D} \Delta S^{0}_{D} \phi_{6} (x) \ dx- \frac{\mu_{0}}{\mu_{1}} \int_{D} \Delta S^{0}_{D} \psi_{6}(x) \ dx +\frac{\mu_{0}}{\mu_{1}} \int_{D} \Delta S^{0}_{D} \psi(x) \ dx \\
&\qquad -\frac{\mu_{0}}{\mu_{1}} \int_{\partial D} \Big[\frac{1}{2} Id+ (K^{\kappa_{1}}_{D})^{*} \Big] \psi(x)\ ds(x)  \\
&=-\frac{\mu_{0}}{\mu_{1}} \int_{\partial D} \Big[\frac{1}{2} Id+ (K^{\kappa_{1}}_{D})^{*} \Big] \psi(x) \ ds(x) \\
&=-\frac{\mu_{0}}{\mu_{1}} \int_{\partial D} \partial_{\nu} S^{\kappa_{1}}_{D} \psi(x)\ ds(x) 
=-\frac{\mu_{0}}{\mu_{1}} \int_{D} \Delta S^{\kappa_{1}}_{D} \psi(x)\ dx 
=\frac{\mu_{0}}{\mu_{1}} \kappa_{1}^{2} \int_{D} S^{\kappa_{1}}_{D} \psi(x) \ dx. 
\end{aligned}
\label{p6-1}
\end{equation}
\begin{lemma}\label{int-psi}
$\int_{D} S^{\kappa_{1}}_{D} \psi(x)\ dx = u^{I}(z) \vert D \vert +
                                                                                                                   \mathcal{O}(\delta^{4}), \ \mbox{ if }\,\mu_1\sim \delta^{-\beta},\ \beta \geq 0,$ and $\kappa_{1}^{2} \sim \delta^{-\alpha},\ 0 \leq \alpha <1 $.
                                                                                                                  
\end{lemma}
\begin{proof}
From the first identity in \eqref{sing-obs-1}, we have
\begin{equation}
\begin{aligned}
(\Delta+\kappa_{0}^{2}) (\underbrace{S^{\kappa_{1}}_{D} \psi- S^{\kappa_{0}}_{D} \phi - u^{I}}_{G}) &= -(\kappa_{0}^{2}-\kappa_{1}^{2}) S^{\kappa_{1}}_{D} \psi \ \text{in}\ D, \\
G&=0 \ \text{on} \ \partial D. 
\end{aligned}
\label{p6-2}
\end{equation}
For $\xi \in B$, let us denote by $\hat{G}(\xi):= G(\delta \xi+z)$. Then $\hat{G}$ satisfies
\begin{equation}
\begin{aligned}
(\Delta+ \kappa_{0}^{2} \delta^{2}) \hat{G} &= -\delta^{2} (\kappa_{0}^{2}-\kappa_{1}^{2}) (S^{\kappa_{1}}_{D} \psi)^{\wedge} =  -(\kappa_{0}^{2}-\kappa_{1}^{2}) \delta^{3} S^{\kappa_{1},\delta}_{B} \hat{\psi} \ \text{in} \ B, \\
\hat{G} &= 0 \ \text{on} \ \partial B,
\end{aligned}
\label{p6-3}
\end{equation}
where $S^{\kappa_{1},\delta}_{B} \hat{\psi}(\xi):= \int_{\partial B} \frac{e^{i \kappa_{1} \delta \vert \xi-\eta \vert}}{4 \pi \vert \xi-\eta \vert} \hat{\psi}(\eta) \ ds(\eta) $.
Multiplying both sides of \eqref{p6-3} by $\hat G$ and integrating by parts, we obtain
\begin{equation}
\int_{B} \vert \nabla \hat{G} \vert^{2} \ d\xi- \delta^{2} \kappa_{0}^{2} \int_{B} \vert \hat{G} \vert^{2} \ d\xi= -(\kappa_{1}^{2}-\kappa_{0}^{2}) \delta^{3} \int_{B} S^{\delta,\kappa_{1}}_{B} \hat{\psi} \hat{G} \ d\xi.
\label{p6-4}
\end{equation}
Using Poincare's inequality, we get
\begin{equation}
\begin{aligned}
C(B) \int_{B} \vert \hat{G} \vert^{2} \ d\xi- \delta^{2} \kappa_{0}^{2} \int_{B} \vert \hat{G} \vert^{2}\ d\xi & \leq \int_{B} \vert \nabla \hat{G} \vert^{2} \ d\xi-\delta^{2} \kappa_{0}^{2} \int_{B} \vert \hat{G} \vert^{2} \ d\xi \\
&= -(\kappa_{1}^{2}-\kappa_{0}^{2}) \delta^{3} \int_{B} S^{\kappa_{1},\delta}_{B} \hat{\psi} \hat{G} \ d\xi.
\end{aligned}
\notag
\end{equation}
Using Holder's inequality on the right hand side, we then have
\begin{equation}
(C(B)-\delta^{2} \kappa_{0}^{2}) \Vert \hat{G} \Vert^{2}_{L^{2}(B)} \leq \vert \kappa_{1}^{2} - \kappa_{0}^{2} \vert \delta^{3} \Vert S^{\kappa_{1},\delta}_{B} \hat{\psi} \Vert_{L^{2}(B)} \Vert \hat{G} \Vert_{L^{2}(B)}.
\notag
\end{equation}
For $\delta>0$ small enough, we have $(C(B)-\delta^{2} \kappa_{0}^{2}) > 0$, and hence
\begin{equation}
\Vert \hat{G} \Vert_{L^{2}(B)} \leq \frac{\vert \kappa_{1}^{2}-\kappa_{0}^{2} \vert}{C(B)-\delta^{2} \kappa_{0}^{2}} \delta^{3} \Vert S^{\kappa_{1},\delta}_{B} \hat{\psi} \Vert_{L^{2}(B)}.
\notag
\end{equation}
For $\xi \in B$,
\begin{equation}
\begin{aligned}
\vert S^{\kappa_{1},\delta}_{B} \hat{\psi}(\xi) \vert &\leq \int_{\partial B} \Big\vert \frac{e^{i \kappa_{1} \delta \vert \xi-\eta \vert}}{4 \pi \vert \xi-\eta \vert} \Big\vert \vert \hat{\psi}(\eta) \vert \ ds(\eta) =\int_{\partial B} \frac{1}{4 \pi \vert \xi-\eta \vert }  \vert \hat{\psi}(\eta) \vert \ ds(\eta) \\
&\leq \Big\Vert \frac{1}{4\pi \vert \xi-. \vert} \Big\Vert_{L^{2}(\partial B)} \Vert \hat{\psi} \Vert_{L^{2}(\partial B)} = \delta^{-1} \Vert \psi \Vert_{L^{2}(\partial D)} \Big\Vert \frac{1}{4\pi \vert \xi-. \vert} \Big\Vert_{L^{2}(\partial B)} \\
&\leq C \delta^{-1} \Big\Vert \frac{1}{4\pi \vert \xi-. \vert} \Big\Vert_{L^{2}(\partial B)} 
\end{aligned}
\notag
\end{equation}
which implies 
\begin{equation}
\Vert S^{\kappa_{1},\delta}_{B} \hat{\psi} \Vert_{L^{2}(B)}=(\int_{B} \vert S^{\kappa_{1},\delta}_{B} \hat{\psi}(\xi) \vert^{2}\ d\xi)^{\frac{1}{2}} \leq C \delta^{-1} \Big( \int_{B} \Big\Vert \frac{1}{4 \pi \vert \xi-. \vert} \Big\Vert_{L^{2}(\partial B)}^{2} \ d\xi \Big)^{\frac{1}{2}} \leq C \delta^{-1}.
\notag
\end{equation}
Hence
\begin{equation}
\Vert \hat{G} \Vert_{L^{2}(B)} \leq C \delta^{2-\alpha}.
\label{p6-5}
\end{equation}
Therefore, we have
\begin{equation}
\Vert G \Vert_{L^{2}(D)} = \delta^{\frac{3}{2}} \Vert \hat{G} \Vert_{L^{2}(B)}= \mathcal{O}(\delta^{\frac{7}{2}-\alpha}).
\label{p6-6}
\end{equation}
By using Holder's inequality, we can write 
\begin{equation}
\int_{D} G \leq \Vert G \Vert_{L^{2}(D)} (\int_{D} 1)^{\frac{1}{2}} = \Vert G \Vert_{L^{2}(D)} \delta^{\frac{3}{2}} = \mathcal{O}(\delta^{5-\alpha}). 
\notag
\end{equation}
 But this implies
 \begin{equation}
 \int_{D} S^{\kappa_{1}}_{D} \psi= \int_{D} S^{\kappa_{0}}_{D} \phi+ \int_{D} u^{I} + \mathcal{O}(\delta^{5-\alpha}).
  \label{p6-7}
 \end{equation}
We next estimate $\int_{D} S^{\kappa_{0}}_{D} \phi(x)\ dx $ and $\int_{D} u^{I}(x)\ dx$.\\
 Using a Taylor series expansion around the centre $z$, we have the estimate
 \begin{equation}
 \int_{D} u^{I}(x) \ dx = u^{I}(z) \vert D \vert+ \mathcal{O}(\delta^{4}).
 \label{p6-8}
 \end{equation}
To estimate  $\int_{D} S^{\kappa_{0}}_{D} \phi(x)\ dx $ we observe that
\begin{align}\label{p6-9}
\vert \int_{D} S^{\kappa_{0}}_{D} \phi(x)\ dx \vert &\leq \int_{D} \int_{\partial D} \frac{\vert \phi(t) \vert}{4 \pi \vert x-t  \vert } \ ds(t) \ dx = \int_{\partial D} \vert \phi(t) \vert (\int_{D} \frac{1}{4 \pi \vert x-t \vert} \ dx) \ ds(t) \nonumber\\
&= \mathcal{O}(\delta^{2}) \int_{\partial D} \vert \phi(t) \vert \ ds(t) \nonumber\\
&\leq \mathcal{O}(\delta^{2}) \vert \partial D \vert^{\frac{1}{2}} \Vert \phi \Vert_{L^{2}(\partial D)}  =                                                                                                                    \mathcal{O}(\delta^{4}).
\end{align}
Using \eqref{p6-8} and \eqref{p6-9} in \eqref{p6-7}, the required result follows.
\end{proof}
Therefore from \eqref{p6-1}, we have
\begin{align}\label{p6-10}
\int_{\partial D} \phi_{6}(x)\ ds(x) = \frac{\mu_{0}}{\mu_{1}} \kappa_{1}^{2} u^{I}(z) \vert D \vert+ \mathcal{O}(\frac{\mu_0}{\mu_1} \kappa_{1}^{2}\delta^{4})=\frac{\mu_{0}}{\mu_{1}} \kappa_{1}^{2} u^{I}(z) \delta^{3} \vert B \vert+\mathcal{O}(\frac{\mu_0}{\mu_1} \kappa_{1}^{2} \delta^{4}).
\end{align}      
From the first identity of \eqref{p6}, we have (using harmonicity and maximum principle and the fact that the left-hand side is a constant)
\begin{equation}
\begin{aligned}
\partial_{\nu} S^{0}_{D} \psi_{6} \Big\vert_{-} - \partial_{\nu} S^{0}_{D} \phi_{6} \Big\vert_{-} &=0 \\
\Rightarrow \Big[\frac{1}{2} Id+ (K^{0}_{D})^{*} \Big] \psi_{6} -\Big[\frac{1}{2} Id+ (K^{0}_{D})^{*} \Big] \phi_{6}&=0.
\end{aligned}
\notag
\end{equation}
Using this, from the second identity, we obtain
\begin{equation}
\begin{aligned}
 &\frac{\mu_{1}}{\mu_{0}} \Big[-\frac{1}{2} Id + (K^{0}_{D})^{*} \Big] \phi_{6}- \Big[\frac{1}{2} Id+ (K^{0}_{D})^{*} \Big] \phi_{6}+ \Big[(K^{0}_{D})^{*}- (K^{\kappa_{1}}_{D})^{*} \Big] \psi(x) = 0\\
&\Rightarrow \frac{\mu_{1}-\mu_{0}}{\mu_{0}} \Big[\lambda Id+ (K^{0}_{D})^{*} \Big] \phi_{6} = \Big[(K^{\kappa_{1}}_{D})^{*}-(K^{0}_{D})^{*} \Big] \psi \\
&\Rightarrow \phi_{6}=\frac{\mu_{0}}{\mu_{1}-\mu_{0}} \Big[\lambda Id+ (K^{0}_{D})^{*} \Big]^{-1} \Big[(K^{\kappa_{1}}_{D})^{*}-(K^{0}_{D})^{*} \Big] \psi.
\end{aligned}
\notag
\end{equation}
This then implies that $\Vert \phi_{6} \Vert_{L^{2}(\partial D)}= \mathcal{O}(\frac{\mu_0}{\mu_1} \delta^{2-\alpha})$ and hence using Cauchy-schwarz inequality, we have
\begin{equation}
\vert \int_{\partial D} (x-z)^{i} \phi_{6} \vert = \mathcal{O}(\frac{\mu_0}{\mu_1} \delta^{4-\alpha}).
\label{p6-11}
\end{equation}
\subsubsection{Treatment of Problem P7} 
Integrating the second identity of \eqref{p7} over $\partial D $, using the fact that $\int_{\partial D} (K^{0}_{D})^{*} \psi_{7}(x)=-\frac{1}{2} \int_{\partial D} \psi_{7}(x) $, we can conclude that
\begin{equation}
\begin{aligned}
&\int_{\partial D} \phi_{7}(x) = \mu_{0} \int_{\partial D} \mathcal{O}(\delta^{3}) \leq C \mu_{0} \delta^{3} \vert \partial D \vert^{\frac{1}{2}} \\
\Rightarrow &\int_{\partial D} \phi_{7}(x) = \mathcal{O}( \delta^{4}).
\end{aligned}
\label{p7-1}
\end{equation}
For functions $ F \in H^{1}(\partial D)$ and $ G \in L^{2}(\partial D)$ in the right-hand side of \eqref{p7}, we can estimate the $L^{2}$-norm of $\phi_{7}$ as follows. From the second identity in \eqref{p7}, we can write 
\begin{equation}
\begin{aligned}
\phi_{7}&= \frac{\mu_{0}}{\mu_{1}} \Big[-\frac{1}{2} Id+(K^{0}_{D})^{*} \Big]^{-1} \Big[\frac{1}{2} Id + (K^{0}_{D})^{*} \Big] \psi_{7} - \mu_{0} \Big[-\frac{1}{2} Id+ (K^{0}_{D})^{*} \Big]^{-1} G.
\end{aligned}
\notag
\end{equation}
Again, from the first identity it follows that
\begin{equation}
\psi_{7}= (S^{0}_{D})^{-1} F + \phi_{7}.
\notag
\end{equation}
Combining the above two, we can write
\begin{equation}
\phi_{7}=\Big[Id-\tilde{A} \Big]^{-1} \tilde{A} (S^{0}_{D})^{-1} F-\mu_{0} \Big[Id-\tilde{A} \Big]^{-1} \Big[-\frac{1}{2} Id+(K^{0}_{D})^{*} \Big]^{-1} G,
\notag
\end{equation}
where $\tilde{A}=\frac{\mu_{0}}{\mu_{1}} \Big[-\frac{1}{2} Id+ (K^{0}_{D})^{*} \Big]^{-1} \Big[\frac{1}{2} Id+(K^{0}_{D})^{*} \Big] $. \\
Now we can use the facts that $\Big[-\frac{1}{2} Id+(K^{0}_{D})^{*} \Big]^{-1}, \Big[\frac{1}{2} Id +(K^{0}_{D})^{*} \Big], \Big[Id-\tilde{A} \Big]^{-1} $ are uniformly bounded with respect to $\delta$  and $(S^{0}_{D})^{-1} $ scales as $\frac{1}{\delta}$ to prove the estimate
\begin{equation}
\Vert \phi_{7} \Vert_{L^{2}(\partial D)} \leq \frac{\mu_{0}}{\mu_{1}} \cdot C \cdot \frac{1}{\delta} \Vert F \Vert_{H^{1}(\partial D)} + \mu_{0} C \Vert G \Vert_{L^{2}(\partial D)}. 
\notag
\end{equation}
Using this in our setup, we obtain
\begin{equation}
\begin{aligned}
\Vert \phi_{7} \Vert_{L^{2}(\partial D)} &\leq C \frac{\mu_{0}}{\mu_{1}} \cdot \frac{1}{\delta} \cdot  \delta^{3} + C \mu_{0} \delta^{3} \leq C \frac{\mu_{0}}{\mu_{1}} \delta^{2} + C \mu_{0} \delta^{3} ,
\end{aligned}
\notag
\end{equation}
and hence
\begin{equation}
\begin{aligned}
\int_{\partial D} (x_{i}-z_{i}) \ \phi_{7}(x) &\leq \Vert x_{i}-z_{i} \Vert_{L^{2}(\partial D)} \Vert \phi_{7} \Vert_{L^{2}(\partial D)}  \leq C \frac{\mu_{0}}{\mu_{1}} \delta^{4}+ C \mu_{0} \delta^{5} \\
\Rightarrow \int_{\partial D} (x_{i}-z_{i}) \ \phi_{7}(x) &= \mathcal{O}(\frac{\mu_{0}}{\mu_{1}} \delta^{4})+\mathcal{O}(\delta^5) .
\end{aligned}
\label{p7-2}
\end{equation} 
\subsubsection{Asymptotics for the near-field and far-field pattern}
In this section, we sum up the results obtained in the previous sections to find the asymptotics for the far field pattern.\\
From the analysis in the previous sections, we have already seen that for $\alpha \in (0,1) $,
\begin{equation}
\begin{aligned}
 &\int_{\partial D} \phi_{1}(x) \ ds(x)=0, \ \int_{\partial D} (x-z)^{j} \phi_{1}(x)\ ds(x)=- \sum_{i=1}^{3} \partial_{i} u^{I}(z) \delta^{3} \int_{\partial B}y^j[\lambda Id+(K_{B}^0)^*]^{-1}(\nu_x\cdot \nabla x^i)(y)\  ds(y), \\
 &\int_{\partial D} \phi_{2}(x)\ ds(x)= - \kappa^{2}_{0} u^{I}(z) \vert B \vert \delta^{3}, \int_{\partial D} (x-z)^{j} \phi_{2}(x) \ ds(x)= \mathcal{O}(\delta^{4}), \\
 &\int_{\partial D} \phi_{3}(x)\ ds(x)=0, \ \int_{\partial D} (x-z)^{j} \phi_{3}(x)\ ds(x)=\mathcal{O}(\frac{\mu_0}{\mu_1}\delta^{4}), \\
 &\int_{\partial D} \phi_{4}(x)\ ds(x)=0,\ \int_{\partial D} (x-z)^{j} \phi_{4}(x)\ ds(x)=0, 
 \end{aligned}
 \label{expansions}
 \end{equation}
 \begin{equation}
 \begin{aligned}
 &\int_{\partial D} \phi_{5}(x)\ ds(x)=0,\ \int_{\partial D} (x-z)^{j} \phi_{5}(x)\ ds(x)=\mathcal{O}(\frac{\mu_0}{\mu_1}\delta^{4-\alpha}) ,\\
 &\int_{\partial D} \phi_{6}(x)\ ds(x)= \frac{\mu_{0}}{\mu_{1}} \kappa_{1}^{2} u^{I}(z) \delta^{3} \vert B \vert+\mathcal{O}(\frac{\mu_0}{\mu_1}\kappa_{1}^{2}\delta^{4}), \ \int_{\partial D} (x-z)^{j} \phi_{6}(x)\ ds(x)=\mathcal{O}(\frac{\mu_0}{\mu_1}\delta^{4-\alpha}), \\ 
&\int_{\partial D} \phi_{7}(x)\ ds(x)=\mathcal{O}(\delta^{4}),\ \int_{\partial D} (x-z)^{j} \phi_{7}(x)\ ds(x)=\mathcal{O}(\frac{\mu_0}{\mu_1}\delta^{4})+\mathcal{O}(\delta^5) ,\ \text{where}\ \vert j \vert=1.
\end{aligned}
\notag
\end{equation}
We recall that the total field is of the form
\begin{equation}
u(x)=\begin{cases}
      u^{I}+ S^{\kappa_{0}}_{D} \phi(x), \ x \in \mathbb{R}^{3}\diagdown  D,   \\
      S^{\kappa_{1}}_{D} \psi(x) , \ x \in D.           \end{cases}
\notag
\end{equation}
Thus the scattered field is given by
\begin{equation}
u^{s}(x)=S^{\kappa_{0}}_{D} \phi(x)= \int_{\partial D} \Phi^{\kappa_{0}}(x,t) \phi(t) \ ds(t).
\notag
\end{equation}
\subsubsection{The near-fields}
To estimate the near fields, we first derive the following estimate 
\begin{equation}\label{Near-fields-formulas-1}
 u^s(x)=\Phi^{\kappa_0}(x, z)\int_{\partial D}\varphi(t)ds(t)+\nabla \Phi^{\kappa_0}(x, z) \cdot \int_{\partial D}(t-z)\varphi(t)ds(t) +\mathcal{O}(d^{-3}(x,z)\vert \int_{\partial D}(t-z)^2\varphi(t)ds(t)\vert) 
\end{equation}
where $d:=d(x, z)$. Then by \eqref{expansions}, we deduce that
\begin{equation}\label{Near-fields-formulas-2}
\begin{aligned}
u^s(x)&=\Phi^{\kappa_0}(x, z)[\omega^2 \mu_0 (\epsilon_1-\epsilon_0)u^I(z)\delta^3 \vert B \vert+ {\mathcal{O}(\delta^{4})}+ \mathcal{O}(\frac{\mu_0}{\mu_1} \kappa_{1}^{2} \delta^4)]\\
&-\nabla \Phi^{\kappa_0}(x, z)\cdot \Big[\sum^{3}_{i=1} \partial_{i} u^{I}(z) \delta^{3} \int_{\partial B}y^j[\lambda Id+(K_{B}^0)^*]^{-1}(\nu_x\cdot \nabla x^i)(y)\  ds(y) +\mathcal{O}(\frac{\mu_{0}}{\mu_{1}} \delta^{4-\alpha})+\mathcal{O}(\delta^4) \Big]+\mathcal{O}(d^{-3}(x,z)\delta^4) \\
&=\Phi^{\kappa_0}(x, z)[\omega^2 \mu_0 (\epsilon_1-\epsilon_0)u^I(z)\delta^3 \vert B \vert+ {\mathcal{O}(\delta^{4})}+ \mathcal{O}(\delta^{4-\alpha+\beta})]\\
&-\nabla \Phi^{\kappa_0}(x, z)\cdot  \left[\sum^{3}_{i=1} \partial_{i} u^{I}(z) \delta^{3} \int_{\partial B}y^j[\lambda Id+(K_{B}^0)^*]^{-1}(\nu_x\cdot \nabla x^i)(y)\  ds(y) +\mathcal{O}(\delta^{4-\alpha+\beta})+\mathcal{O}(\delta^4) \right] +\mathcal{O}(d^{-3}(x,z) \delta^4).
\end{aligned}
\end{equation}
\subsubsection{The far-fields}
The far field pattern $u^{\infty}(\hat{x})$ can be written as
\begin{equation}\label{Farfields-formulas}
\begin{aligned}
u^{\infty}(\hat{x}) &= \int_{\partial D} e^{- i \kappa_{0} \hat{x}.t} \phi(t)\ ds(t)\\ 
&= e^{- i \kappa_{0} \hat{x}.z} \int_{\partial D} \phi(t) \ ds(t)+ \nabla_{z} e^{-i \kappa_{0} \hat{x}.z} \int_{\partial D} (t-z) \phi(t)\ ds(t)+ \mathcal{O}\Big(\int_{\partial D} \vert t-z \vert^{2} \vert \phi(t) \vert \ ds(t) \Big).
\end{aligned}
\notag
\end{equation}
Therefore, 
%
\begin{equation}
\begin{aligned}
 u^{\infty}(\hat{x})&=e^{- i \kappa_{0} \hat{x}.z} \int_{\partial D} \phi(t) \ ds(t)+ \nabla_{z} e^{-i \kappa_{0} \hat{x}.z} \int_{\partial D} (t-z) \phi(t)\ ds(t)+ \mathcal{O}(\delta^{4}) \\
&= e^{-i \kappa_{0} \hat{x}.z} \Big[\frac{\mu_{0}}{\mu_{1}} \kappa_{1}^{2} u^{I}(z) \delta^{3} \vert B \vert- \kappa_{0}^{2} u^{I}(z) \delta^{3} \vert B \vert+\mathcal{O}(\frac{\mu_0}{\mu_1}k_{1}^{2}\delta^{4}) +\mathcal{O}(\delta^4)\Big] \\
&\quad -\nabla_{z} e^{- i \kappa_{0} \hat{x}.z} \Big(\sum^{3}_{i=1} \partial_{i} u^{I}(z) \delta^{3} \int_{\partial B}y^j[\lambda Id+(K_{B}^0)^*]^{-1}(\nu_x\cdot \nabla x^i)(y)\  ds(y) +\mathcal{O}(\frac{\mu_0}{\mu_1}\delta^{4-\alpha}) +\mathcal{O}(\delta^{4}) \Big)\\
&\qquad + \mathcal{O}(\delta^{4})\\
 &= e^{-i \kappa_{0} \hat{x}.z} \Big[ \omega^{2} \mu_{0} (\epsilon_{1}-\epsilon_{0}) u^{I}(z) \delta^{3} \vert B \vert+\mathcal{O}(\delta^{4-\alpha+\beta})+\mathcal{O}(\delta^4)   \Big] \\
&\quad -\nabla_{z} e^{- i \kappa_{0} \hat{x}.z} \Big[ \Big(\sum^{3}_{i=1} \partial_{i} u^{I}(z) \delta^{3}  \int_{\partial B}y^j[\lambda Id+(K_{B}^0)^*]^{-1}(\nu_x\cdot \nabla x^i)(y)\  ds(y) \Big)+\mathcal{O}(\delta^{4-\alpha+\beta})+\mathcal{O}(\delta^4) \Big] +\mathcal{O}(\delta^{4}).
 \end{aligned}
\notag
\end{equation}

\subsection{Justifying the Expansions for variable $\epsilon_0$}\label{sec-variable epsilon}
Let $\epsilon_0$ be variable with $\epsilon(x)=\epsilon_0$ for $x\in\mathbb{R}^3\backslash\bar{\Omega}$, where $\Omega$ is a bounded and smooth domain. We also assume that $\mu_0$ to be constant in $\mathbb{R}^3$. Hence, $\kappa_0(x)=\omega\sqrt{\epsilon(x)\mu_0}$ varies in $\Omega$ and $\kappa_0(x)=\kappa_0=constant$ for $x\in\mathbb{R}^3\backslash\bar{\Omega}$. 
\subsubsection{Comparison of estimates to the homogeneous background case }
We shall now prove the estimates on $\Vert S^{\kappa_{0}}_{D} -S^{0}_{D}  \Vert_{\mathcal{L}(L^{2}(\partial D),H^{1}(\partial D))} $ and $\Vert (K^{\kappa_{0}}_{D})^{*}- (K^{0}_{D})^{*}  \Vert_{\mathcal{L}(L^{2}(\partial D),L^{2}(\partial D))} $.\\
In contrast to the case of homogeneous background, we are able to derive slightly weaker estimates in this case.
\begin{lemma}
$\Vert S^{\kappa_{0}}_{D} -S^{0}_{D}  \Vert_{\mathcal{L}(L^{2}(\partial D),H^{1}(\partial D))} = \mathcal{O}(\delta^{2} \vert log \ \delta \vert)$, $\Vert (K^{\kappa_{0}}_{D})^{*}- (K^{0}_{D})^{*}  \Vert_{\mathcal{L}(L^{2}(\partial D),L^{2}(\partial D))}= \mathcal{O}(\delta^{2} \vert log \ \delta \vert) $. 
\end{lemma} 
\begin{proof}
Let us define $G(x,z):= G^{\kappa_{0}}(x,z)-\Phi(x,z), $ where $\Phi(x,z):= \frac{1}{4 \pi \vert x-z \vert}$. We know that $\Phi(\cdot,z) \in L^{p}_{\text{loc}}(\mathbb{R}^{3}) $ for any $p<3 $. Also $(\Delta+ \kappa^{2}_{0}(x)) G^{\kappa_{0}}(x,z)= - \delta(z) $.\\
Let $B $ be a large ball such that $\Omega \Subset B $ and we consider the equation
\begin{equation}
(\Delta+\kappa_{0}^{2}(x))G(\cdot,z)=-\kappa_{0}^{2}(x) \Phi(\cdot,z) \ \text{in}\ B.
\label{e500}
\end{equation}
Since $G\Big\vert_{\partial B}(\cdot,z) $ is uniformly bounded for $z \in \Omega $, by interior regularity we can conclude that $G(\cdot,z) \in H^{2}(\Omega) $ and $\Vert G(\cdot,z) \Vert_{H^{2}(\Omega)} $ is uniformly bounded for $z \in \Omega $. By Sobolev embedding, it follows that $\Vert G(\cdot,z) \Vert_{C(\Omega)} $ is bounded uniformly for $z \in \Omega $, whence we can conclude that
\begin{equation}
\Vert S^{\kappa_{0}}_{D} \phi - S^{0}_{D} \phi \Vert_{L^{2}(\partial D)} \leq C \delta^{2} \Vert \phi \Vert_{L^{2}(\partial D)}.
\notag
\end{equation}
To estimate the $H^{1}$ norm, we similarly consider the equation
\begin{equation}
(\Delta+ \kappa_{0}^{2}(x)) \partial_{j} G(x,z)= -\partial_{j}[\kappa^{2}_{0}(x)] (G(x,z)+\Phi(x,z))- \kappa_{0}^{2}(x) \partial_{j} \Phi(x,z) \ \text{in} \ B. 
\label{e501}
\end{equation}
We split this problem into the following two sub-problems
\begin{equation}
\begin{aligned}
(\Delta+ \kappa_{0}^{2}(x)) H(x,z)&= -\partial_{j}[\kappa^{2}_{0}(x)] (G(x,z)+\Phi(x,z)) \ \text{in} \ B,\\
H(x,z)&= \partial_{j}G(x,z)\ \text{on}\ \partial B,
\end{aligned}
\label{e502-0}
\end{equation}
and 
\begin{equation}
\begin{aligned}
(\Delta+ \kappa_{0}^{2}(x)) F(x,z) &=  -\kappa_{0}^{2}(x) \partial_{j} \Phi(x,z) \ \text{in} \ B,\\
F &= 0 \ \text{on} \ \partial B.
\end{aligned}
\label{e502}
\end{equation}
The problem \eqref{e502-0} can be dealt with just as in the previous case recalling that $k_0$ is Lipschitz continuous. To study the problem
\eqref{e502}, we proceed as follows.\\
Let $\Phi_{B} $ denote the Green's function corresponding to \eqref{e502}. Then
\begin{equation}
\begin{aligned}
&F(x,z)= -\int_{B} \kappa_{0}^{2}(y) \partial_{j} \Phi(y,z) \Phi_{B}(y,x) \ dy \\
\Rightarrow \vert F(x,z) \vert &\leq C \int_{B} \frac{1}{\vert y-z \vert^{2}} \cdot \frac{1}{\vert x-y \vert} \ dy\leq C\  log(\frac{1}{\vert x-z \vert})+C.
\end{aligned}
\notag
\end{equation} 
Using Minkowski's inequality and Cauchy-Schwarz inequality, we can write
\begin{equation}
\begin{aligned}
\Vert \int_{\partial D} F(x,z) \phi(z) \ ds(z) \Vert_{L^{2}(\partial D)} &= \Big(\int_{\partial D} \Big\vert \int_{\partial D} F(x,z) \ \phi(z) \ ds(z) \Big\vert^{2}\ ds(x) \Big)^{\frac{1}{2}}\\
& \leq \int_{\partial D} \Big(\int_{\partial D} \vert F(x,z) \phi(z) \vert^{2} \ ds(x) \Big)^{\frac{1}{2}} ds(z)\\
&\leq \int_{\partial D} \Big(\int_{\partial D} \vert F(x,z) \vert^{2} ds(x) \Big)^{\frac{1}{2}} \vert \phi(z) \vert \ ds(z)\\
& \leq \Big(\int_{\partial D} \int_{\partial D} \vert F(x,z) \vert^{2} \ ds(x) ds(z) \Big)^{\frac{1}{2}} \Vert \phi \Vert_{L^{2}(\partial D)} .
\end{aligned}
\notag
\end{equation}
Since $\vert F(x,z) \vert \leq C log(\frac{1}{\vert x-z \vert}) + C $, the dominant order term can be estimated by the behaviour of the term $log (\frac{1}{\vert x-z \vert}) $. Now
\begin{equation}
\begin{aligned}
\int_{\partial D} \int_{\partial D} \Big\vert log \ \frac{1}{\vert x-z \vert} \Big\vert^{2} \ ds(x) ds(z) &=\int_{\partial D} \int_{\partial D} \Big\vert log \ \vert x-z \vert \Big\vert^{2} \ ds(x) ds(z)\\
&= \int_{\partial B} \int_{\partial B} \Big\vert log \ \delta \vert \hat{x}-\hat{z} \vert \Big\vert^{2} \delta^{2} \delta^{2} \ ds(\hat{x}) ds(\hat{z}) \\
&= \delta^{4} \int_{\partial B} \int_{\partial B} \Big\vert log\ \delta + log \vert \hat{x}-\hat{z} \vert \Big\vert^{2} \ ds(\hat{x}) ds(\hat{z}) = \mathcal{O}(\delta^{4} (log \ \delta)^{2} ),
\end{aligned}
\notag
\end{equation}
and therefore $\Vert \int_{\partial D} F(x,z) \phi(z) \ ds(z) \Vert_{L^{2}(\partial D)} \leq C \Vert \phi \Vert_{L^{2}(\partial  D)} \delta^{2} \vert log \ \delta \vert $.\\
Summing up the discussion above, we can therefore conclude that $\Vert S^{\kappa_{0}}_{D}-S^{0}_{D} \Vert_{\mathcal{L}(L^{2}(\partial D),H^{1}(\partial D))} = \mathcal{O}(\delta^{2} \vert log \delta \vert)  $. 
The proof for the other case follows similarly.
\end{proof}
\subsubsection{Deriving the far-field expansions:}
To estimate $\int_{\partial D}\phi$ and $\int_{\partial D}(y-z)\phi(y) ds(y)$, we  split the problem 
\begin{equation}
\begin{aligned}
&S^{\kappa_{1}}_{D} \psi- S^{\kappa_{0}}_{D} \phi = U^{t},\\
 &\frac{1}{\mu_{1}} \Big[\frac{1}{2} Id+(K^{\kappa_{1}}_{D})^{*}\Big] \psi -\frac{1}{\mu_{0}} \Big[-\frac{1}{2} Id + (K^{\kappa_{0}}_{D})^{*} \Big]\phi =\frac{1}{\mu_{0}} \frac{\partial U^{t}}{\partial \nu^1},
 \end{aligned}
 \notag
 \end{equation}
into 7 sub problems as in P1, P2, P3, P4, P5, P6, P7 of our previous computations.\\
For this, we first assume that $\epsilon_0$ is of class $C^{0,1}(\mathbb{R}^3)$. Then, we can show that $U^t\in C^{2, 1}(\mathbb{R}^3)$. Hence we can expand $U^t$ at the order 2 near $z \in\Omega$, with the expansions at hand, we can state the 7 subproblems as in P1, P2, P3, P4, P5, P6, P7 by replacing $u^I(z)$ by $U^t(z)$ in P1, P2, P3, P4 and appropriate $\psi$.\\
 It is clear that the new P1, P2, P3, P4 can be dealt in the same way as the ones before (only replacing $U^{I}(z)$ by $U^t(z)$). Hence, we have the corresponding estimates as derived before.\\
 Let us now discuss the way to handle P5, P6 and P7.
 \begin{itemize}
 \item The analysis of the sub problem P5 use $\psi$ only through its apriori estimates. Hence the proof goes as it is for our case here.
 \item In case of P6, the identity \eqref{p6-1} does not change, i.e., \[ \int_{\partial D_1}\phi_6(x) ds(x) =\frac{\mu_{0}}{\mu_{1}} \kappa_{1}^{2} \int_{D} S^{\kappa_{1}}_{D} \psi(x) \ dx.\]
 To derive the corresponding Lemma \ref{int-psi}, i.e.,
$\int_{D} S^{\kappa_{1}}_{D} \psi(x)\ dx = u^{I}(z) \vert D \vert + \mathcal{O}(\frac{\mu_0}{\mu_1}\delta^{4}),$ we need to change in \eqref{p6-3} by 
$\hat{\kappa}_0(\xi):=\kappa_0(\delta\xi+z)=\kappa_0(z)+O(\delta)$.\\
Hence, $\frac{1}{2}\kappa_0^2(z)\leq\hat{\kappa}_0^2(\xi)\leq\frac{3}{2}\kappa_0^2(z)$ if $\delta$ is small enough. With this property, all the remaining steps go almost the same.

\item We can deal with P7 just as in the case when $k_{0}$ is constant, but now $\int_{\partial D} \phi_{7}(x)= \mathcal{O} ( \delta^{4} \vert log \ \delta \vert) $ 
and $\int_{\partial D} (x_{i}-z_{i}) \phi_{7}(x) = \mathcal{O}(\frac{\mu_{0}}{\mu_{1}} \delta^{4} \vert log \ \delta \vert)+\mathcal{O}(\delta^5 \vert log \ \delta \vert)  $. 
 \end{itemize}
The near fields are then of the form
%
\begin{equation}\label{Near-fields-formulas-V-2}
\begin{aligned}
V^s(x, d)-U^s(x, d)&=G^{\kappa_0}(x, z)[\omega^2 \mu_0 (\epsilon_1-\epsilon_0)U^t(z,d)\delta^3 \vert B \vert+ \mathcal{O}(\delta^{4} \vert log \ \delta \vert)+\mathcal{O}(\frac{\mu_0}{\mu_1}k_{1}^{2}\delta^4)]\\ 
&-\nabla G^{\kappa_0}(x, z)\cdot \Big[( M \nabla U^t(z,d))\delta^3 + \mathcal{O}(\frac{\mu_{0}}{\mu_{1}} \delta^{4} \vert log \ \delta \vert)+\mathcal{O}(\delta^4)+\mathcal{O}(\frac{\mu_0}{\mu_1} \delta^{4-\alpha}) \Big]
+\mathcal{O}(d^{-3}(x,z) \delta^{4}) \\
&=G^{\kappa_0}(x, z)[\omega^2 \mu_0 (\epsilon_1-\epsilon_0)U^t(z,d)\delta^3 \vert B \vert+ \mathcal{O}(\delta^{4} \vert log \ \delta \vert)+\mathcal{O}(\delta^{4-\alpha+\beta}) ]\\ 
&-\nabla G^{\kappa_0}(x, z)\cdot ( M \nabla U^t(z,d))\delta^3 + \mathcal{O}(d^{-2}(x,z) \delta^{4+\beta} \vert log \ \delta \vert) +\mathcal{O}(d^{-2}(x,z) \delta^{4-\alpha+\beta})
+\mathcal{O}(d^{-3}(x,z) \delta^{4}),
\end{aligned}
\end{equation}
if $\mu_1\sim \delta^{-\beta}, \; \beta \geq 0$, where $M:= (M_{ij})_{i,j=1}^{3}$ is defined by $M_{ij}:= \int_{\partial B}y^j[\lambda Id+(K_{B}^0)^*]^{-1}(\nu_x\cdot \nabla x^i)(y)\  ds(y) $.\\
Hence in the case of electric nanoparticles $(\beta=0, 0<\alpha<1) $, we have
\begin{equation}
\begin{aligned}
V^s(x, d)-U^s(x, d)&=G^{\kappa_0}(x, z)[\omega^2 \mu_0 (\epsilon_1-\epsilon_0)U^t(z,d)\delta^3 \vert B \vert+ \mathcal{O}(\delta^{4} \vert log \ \delta \vert) ]\\ 
&-\nabla G^{\kappa_0}(x, z)\cdot ( M \nabla U^t(z,d))\delta^3 +\mathcal{O}(d^{-2}(x,z) \delta^{4-\alpha})
+\mathcal{O}(d^{-3}(x,z) \delta^{4}),
\end{aligned}
\notag
\end{equation}
and in the case of magnetic nanoparticles $(\alpha=\beta, 0<\alpha<1) $, we have
\begin{equation}
\begin{aligned}
V^s(x, d)-U^s(x, d)&=G^{\kappa_0}(x, z)[\omega^2 \mu_0 (\epsilon_1-\epsilon_0)U^t(z,d)\delta^3 \vert B \vert+ \mathcal{O}(\delta^{4} \vert log \ \delta \vert) ]\\ 
&-\nabla G^{\kappa_0}(x, z)\cdot ( M \nabla U^t(z,d))\delta^3 + \mathcal{O}(d^{-2}(x,z) \delta^{4+\alpha} \vert log \ \delta \vert) +\mathcal{O}(d^{-3}(x,z) \delta^{4}).
\end{aligned}
\notag
\end{equation}
The far field pattern $V^{\infty}(\hat{x}, d)$ has the expansion
%
\begin{equation}
\begin{aligned}
V^{\infty}(\hat{x}, d)-U^{\infty}(\hat{x}, d) &= \omega^{2} \mu_{0} (\epsilon_{1}-\epsilon_{0}) U^{t}(z, d)U^t(z,-\hat{x}) \delta^{3} \vert B \vert + U^{t}(z,-\hat{x}) \mathcal{O}(\delta^{4} \vert log \ \delta \vert)+U^{t}(z,-\hat{x}) \mathcal{O}(\delta^{4-\alpha+\beta})\\
&\quad -\nabla_{z} U^t(z,-\hat{x}) \Big(\sum^{3}_{i=1} \partial_{i} U^{t}(z, d) \delta^{3} \int_{\partial B}y^j[\lambda Id+(K_{B}^0)^*]^{-1}(\nu_x\cdot \nabla x^i)(y)\  ds(y)+\mathcal{O}(\frac{\mu_{0}}{\mu_{1}} \delta^{4} \vert log \ \delta \vert)\\
&\qquad+\mathcal{O}(\delta^{4-\alpha+\beta})+\mathcal{O}(\delta^4)\Big) +\mathcal{O}(\delta^{4}), \ \mbox{ if } \mu_1\sim \delta^{-\beta},\; \beta \geq 0. 
\end{aligned}
\label{Farfields-formulas-V}
\end{equation}
%
Therefore in the case of electric nanoparticles $(\beta=0, 0<\alpha<1) $, we have
\begin{equation}
\begin{aligned}
V^{\infty}(\hat{x}, d)-U^{\infty}(\hat{x}, d) &= \omega^{2} \mu_{0} (\epsilon_{1}-\epsilon_{0}) U^{t}(z, d)U^t(z,-\hat{x}) \delta^{3} \vert B \vert \\
&\quad -\nabla_{z} U^t(z,-\hat{x}) \Big(\sum^{3}_{i=1} \partial_{i} U^{t}(z, d) \delta^{3} \int_{\partial B}y^j[\lambda Id+(K_{B}^0)^*]^{-1}(\nu_x\cdot \nabla x^i)(y)\  ds(y) \Big) +\mathcal{O}(\delta^{4-\alpha}),
\end{aligned}
\notag
\end{equation}
and in the case of magnetic nanoparticles $(\alpha=\beta, 0<\alpha<1) $, we have
\begin{equation}
\begin{aligned}
V^{\infty}(\hat{x}, d)-U^{\infty}(\hat{x}, d) &= \omega^{2} \mu_{0} (\epsilon_{1}-\epsilon_{0}) U^{t}(z, d)U^t(z,-\hat{x}) \delta^{3} \vert B \vert \\
&\quad -\nabla_{z} U^t(z,-\hat{x}) \Big(\sum^{3}_{i=1} \partial_{i} U^{t}(z, d) \delta^{3} \int_{\partial B}y^j[\lambda Id+(K_{B}^0)^*]^{-1}(\nu_x\cdot \nabla x^i)(y)\  ds(y) \Big) +\mathcal{O}(\delta^{4} \vert log \ \delta \vert).
\end{aligned}
\notag
\end{equation}

\end{document}